\documentclass{amsart}

\usepackage{mathrsfs}
\usepackage{stmaryrd}
\usepackage{enumerate}
\usepackage{aliascnt}
\usepackage[colorlinks=true, linkcolor=black, citecolor=magenta]{hyperref}

\newtheorem{theorem}{Theorem}

\newaliascnt{lemma}{theorem}
\newtheorem{lemma}[lemma]{Lemma}
\aliascntresetthe{lemma}

\newaliascnt{corollary}{theorem}
\newtheorem{corollary}[corollary]{Corollary}
\aliascntresetthe{corollary}

\newaliascnt{proposition}{theorem}
\newtheorem{proposition}[proposition]{Proposition}
\aliascntresetthe{proposition}

\newaliascnt{remark}{theorem}
\newtheorem{remark}[remark]{Remark}
\aliascntresetthe{remark}

\newaliascnt{fact}{theorem}
\newtheorem{fact}[fact]{Fact}
\aliascntresetthe{fact}

\theoremstyle{definition}

\newtheorem*{acknowledgements}{Acknowledgements}

\newtheorem*{notation}{Notation}

\newaliascnt{definition}{theorem}
\newtheorem{definition}[definition]{Definition}
\aliascntresetthe{definition}

\newcommand{\Z}{\mathbb{Z}}
\newcommand{\R}{\mathbb{R}}
\newcommand{\C}{\mathbb{C}}

\newcommand{\bG}{\mathbb{G}}
\newcommand{\bA}{\mathbb{A}}

\newcommand{\bL}{\mathbb{L}}

\newcommand{\cO}{\mathcal{O}}

\DeclareMathOperator{\Spec}{Spec}

\DeclareMathOperator{\Trop}{Trop}
\DeclareMathOperator{\relint}{relint}

\DeclareMathOperator{\init}{in}
\DeclareMathOperator{\supp}{supp}

\newcommand{\Var}{\mathbf{Var}}

\newcommand{\cM}{\mathcal{M}}

\newcommand{\cB}{\mathcal{B}}

\newcommand{\cA}{\mathcal{A}}

\newcommand{\Gr}{\mathrm{Gr}}

\DeclareMathOperator{\Proj}{Proj}

\DeclareMathOperator{\Hom}{Hom}

\title[]{Hyperplane arrangements and mixed Hodge numbers of the Milnor fiber}
\author{Max Kutler and Jeremy Usatine}

\address{Max Kutler, Department of Mathematics, University of Kentucky}
\email{max.kutler@uky.edu}

\address{Jeremy Usatine, Department of Mathematics, Yale University}
\email{jeremy.usatine@yale.edu}

\begin{document}
\maketitle

\begin{abstract}
For each complex central essential hyperplane arrangement $\mathcal{A}$, let $F_{\mathcal{A}}$ denote its Milnor fiber. We use Tevelev's theory of tropical compactifications to study invariants related to the mixed Hodge structure on the cohomology of $F_{\mathcal{A}}$. We prove that the map taking each arrangement $\mathcal{A}$ to the Hodge-Deligne polynomial of $F_{\mathcal{A}}$ is locally constant on the realization space of any loop-free matroid. When $\mathcal{A}$ consists of distinct hyperplanes, we also give a combinatorial description for the homotopy type of the boundary complex of any simple normal crossing compactification of $F_{\mathcal{A}}$. As a direct consequence, we obtain a combinatorial formula for the top weight cohomology of $F_{\mathcal{A}}$, recovering a result of Dimca and Lehrer.
\end{abstract}

\numberwithin{theorem}{section}
\numberwithin{lemma}{section}
\numberwithin{corollary}{section}
\numberwithin{proposition}{section}
\numberwithin{remark}{section}
\numberwithin{fact}{section}
\numberwithin{definition}{section}

\section{Introduction}

Let $H_1, \dots, H_n$ be a central essential arrangement of hyperplanes in $\C^d$, and let $f_1, \dots, f_n$ be linear forms defining $H_1, \dots, H_n$, respectively. Then the \emph{Milnor fiber} $F$ of the arrangement is defined to be the subvariety of $\C^d$ defined by $f_1 \cdots f_n - 1$. Because $f_1 \cdots f_n$ is homogeneous, a result of Milnor implies that $F$ is diffeomorphic to the topological Milnor fiber of $f_1 \cdots f_n$ at the origin \cite{Milnor}. There has been much interest in understanding how the invariants that arise in singularity theory, such as the Milnor fiber $F$, vary as hyperplane arrangements vary with fixed combinatorial type. A major open conjecture predicts that the Betti numbers of $F$ are combinatorial invariants. Randell has shown that the diffeomorphism type of $F$ remains constant in smooth families of a fixed combinatorial type \cite{Randell}. Budur and Saito proved that a related invariant, the Hodge spectrum, is a combinatorial invariant \cite{BudurSaito}. On the other hand, Walther has recently shown that the Bernstein-Sato polynomial is not a combinatorial invariant \cite{Walther}. Dimca and Lehrer have also recently studied the mixed Hodge structure on the cohomology of $F$ \cite{DimcaLehrer2012, DimcaLehrer2016}. We refer to \cite{Suciu} for a survey on related topics.

In this paper, we use Tevelev's theory of tropical compactifications \cite{Tevelev} to study invariants related to the mixed Hodge structure on the cohomology of $F$. We show that the Hodge-Deligne polynomial of $F$ remains constant as we vary the arrangement $H_1, \dots, H_n$ within the same connected component of a matroid's realization space. We also give a combinatorial description for the homotopy type of any boundary complex of $F$, and this gives a combinatorial formula for the top weight cohomology of $F$.

\subsection{Statement of main results}

Throughout this paper, $k$ will be an algebraically closed field.

Let $d, n \in \Z_{>0}$, let $\mu_n \subset k^\times$ be the group of $n$-th roots of unity, and let $\Gr_{d,n}$ be the Grassmannian of $d$-dimensional linear subspaces in $\bA_k^n = \Spec(k[x_1, \dots, x_n])$. For each $\cA \in \Gr_{d,n}(k)$, let $X_\cA \subset \bA_k^n$ be the corresponding linear subspace, and let $F_\cA$ be the scheme theoretic intersection of $X_\cA$ with the closed subscheme of $\bA_k^n$ defined by $(x_1 \cdots x_n - 1)$. Endow $F_\cA$ with the restriction of the $\mu_n$-action on $\bA_k^n$ where each $\xi \in \mu_n$ acts by scalar multiplication. If $X_\cA$ is not contained in a coordinate hyperplane of $\bA_k^n$, then the restriction of the coordinates $x_i$ define a central essential hyperplane arrangement in $X_\cA$, and $F_\cA$ with its $\mu_n$-action is that arrangement's Milnor fiber with its monodromy action.

\begin{remark}
If $H_1, \dots, H_n$ is a central essential arrangement of hyperplanes in $\bA_k^d$, any choice of linear forms defining the $H_i$ defines a linear embedding of $\bA_k^d$ into $\bA_k^n$, and $H_1, \dots, H_n$ is the hyperplane arrangement associated to the resulting linear subspace of $\bA_k^n$. Thus we lose no generality by studying the hyperplane arrangements associated to $d$-dimensional linear subspaces of $\bA_k^n$.
\end{remark}

Let $\cM$ be a rank $d$ loop-free matroid on $\{1, \dots, n\}$, and let $\Gr_\cM \subset \Gr_{d,n}$ be the locus parametrizing linear subspaces whose associated hyperplane arrangements have combinatorial type $\cM$. Our first main result concerns how certain additive invariants of $F_\cA$ vary as $\cA$ varies within a connected component of $\Gr_\cM$. By \emph{additive invariants}, we mean invariants that satisfy the cut and paste relations.

Let $K_0^{\mu_n}(\Var_k)$ be the $\mu_n$-equivariant Grothendieck ring of $k$-varieties, let $\bL \in K_0^{\mu_n}(\Var_k)$ be the class of $\bA_k^1$ with trivial $\mu_n$-action, and for any separated finite type $k$-scheme $Y$ with good $\mu_n$-action, let $[Y, \mu_n] \in K_0^{\mu_n}(\Var_k)$ denote the class of $Y$ with its $\mu_n$-action. Let $\Z[\bL]$ denote the polynomial ring over the symbol $\bL$, and endow $K_0^{\mu_n}(\Var_k)$ with the $\Z[\bL]$-algebra structure given by $\bL \mapsto \bL$.

\begin{remark}
We mention another interpretation of the additive invariants of $F_\cA$ with its $\mu_n$-action. We note that by additive invariants of a $k$-variety with $\mu_n$-action, we mean those invariants defined by group homomorphisms from the $\mu_n$-equivariant Grothendieck ring of $k$-varieties. Let $k$ have characteristic 0, and let $K_0^{\hat\mu}(\Var_k) = \varinjlim_\ell K_0^{\mu_\ell}(\Var_k)$. Then \cite[Theorem 4.1.1]{NicaisePayne} implies that the canonical inclusion $K_0^{\mu_n}(\Var_k) \hookrightarrow K_0^{\hat\mu}(\Var_k)$ takes $[F_\cA, \mu_n]$ to the motivic nearby fiber of the arrangement associated to $\cA$. Therefore the additive invariants of $[F_\cA, \mu_n]$ are in fact determined by the motivic nearby fiber of the arrangement.
\end{remark}

\begin{definition}
Let $P$ be a $\Z[\bL]$-module, and let $\nu: K_0^{\mu_n}(\Var_k) \to P$ be a $\Z[\bL]$-module morphism. We say that $\nu$ is \emph{constant on smooth projective families with $\mu_n$-action} if the following always holds.
\begin{itemize}
\item If $S$ is a connected separated finite type $k$-scheme with trivial $\mu_n$-action and $X \to S$ is a $\mu_n$-equivariant smooth projective morphism from a scheme $X$ with $\mu_n$-action, then the map $S(k) \to P: s \mapsto \nu[X_s, \mu_n]$ is constant, where $X_s$ denotes the fiber of $X \to S$ over $s$.
\end{itemize}
\end{definition}

We can now state our first main result and its direct corollary.

\begin{theorem}
\label{additiveinvariantsmilnorfiber}
Let $P$ be a torsion-free $\Z[\bL]$-module, let $\nu: K_0^{\mu_n}(\Var_k) \to P$ be a $\Z[\bL]$-module morphism that is constant on smooth projective families with $\mu_n$-action, and assume that the characteristic of $k$ does not divide $n$. 

If $\cA_1, \cA_2 \in \Gr_\cM(k)$ are in the same connected component of $\Gr_\cM$, then
\[
	\nu[F_{\cA_1}, \mu_n] = \nu[F_{\cA_2}, \mu_n].
\]
\end{theorem}

\begin{corollary}
If $k = \C$ and $\cA_1, \cA_2 \in \Gr_\cM(k)$ are in the same connected component of $\Gr_\cM$, then the Hodge-Deligne polynomial of $F_{\cA_1}$ is equal to the Hodge-Deligne polynomial of $F_{\cA_2}$.
\end{corollary}

The varieties $F_\cA$ are not compact, so it is not immediately clear how the Hodge-Deligne polynomials of the $F_\cA$ vary in families. To address this difficulty, we show that certain tropical compactifications of the $F_\cA$ can be constructed in families. A similar idea was used in \cite[Lemma 2.13]{EsterovTakeuchi} to study the virtual Hodge numbers of non-degenerate complete intersections. We extend the $\mu_n$-action on $F_\cA$ to the boundary of these compactifications, and we compare the $\mu_n$-equivariant classes of the boundary strata to classes of Milnor fibers of related hyperplane arrangements. We then obtain \autoref*{additiveinvariantsmilnorfiber} by inducting on the number of bases in $\cM$.

In a forthcoming paper, we use \autoref*{additiveinvariantsmilnorfiber} to obtain a similar result for specializations of the Denef-Loeser motivic zeta functions of hyperplane arrangements.

Our next main result concerns the boundary complexes of compactifications of $F_\cA $ in the case where $\cA \in \Gr_\cM(k)$ and $\cM$ has no pairs of distinct parallel elements. Note that this is the case where the associated hyperplane arrangement consists of distinct hyperplanes, so in particular, $F_\cA$ is irreducible.

If $X$ is an irreducible $k$-variety and $X \subset \overline{X}$ is a simple normal crossing compactification, i.e., an open immersion into a smooth irreducible proper $k$-variety such that $\partial \overline{X} = \overline{X} \setminus X$ is a simple normal crossing divisor, then let $\Delta(\partial{\overline{X}})$ denote the dual complex of the boundary $\partial \overline{X}$. See for example \cite[Section 2]{Payne}. Thuillier proved that the homotopy type of $\Delta(\partial\overline{X})$ does not depend on the choice of simple normal crossing compactification, as long as such a compactification exists \cite{Thuillier}.

Let $\mu(\cM)$ denote the M\"{o}bius number of $\cM$. The M\"{o}bius number is equal to the absolute value of the constant term of the characteristic polynomial of $\cM$, and in particular, it depends only on the matroid $\cM$. We also refer to \cite[Chapter 4]{MaclaganSturmfels} for some equivalent definitions of the M\"{o}bius number $\mu(\cM)$.

\begin{theorem}
\label{boundarycomplexmilnorfiber}
Suppose that $\cM$ has no pairs of distinct parallel elements, let $\cA \in \Gr_\cM(k)$, and assume that the characteristic of $k$ does not divide $n$. 

Then $F_\cA$ admits a simple normal crossing compactification. Furthermore, if $F_\cA \subset \overline{F_\cA}$ is a simple normal crossing compactification, then the boundary complex $\Delta(\partial \overline{F_\cA})$ is homotopy equivalent to a wedge of $\mu(\cM)$ $(d-2)$-dimensional spheres.
\end{theorem}

Because the homotopy type of $\Delta(\partial \overline{F})$ does not depend on the choice of $\overline{F}$, it suffices to find one simple normal crossing compactification whose boundary has dual complex with the desired homotopy type. We show that a certain tropical compactification of $F_\cA$ is smooth and has simple normal crossing boundary, and we show that this tropical compactification has boundary complex homeomorphic to the so-called Bergman complex of $\cM$. We then use Ardila and Klivan's computation of the homotopy type of the Bergman complex \cite{ArdilaKlivans} to obtain \autoref*{boundarycomplexmilnorfiber}.

If $k = \C$, the top weight cohomology of $F_\cA$ with rational coefficients, in the sense of mixed Hodge theory, can be computed in terms of the reduced homology of $\Delta(\partial \overline{F_\cA})$ with rational coefficients. See for example \cite[Theorem 3.1]{Hacking}. We thus get the following corollary of \autoref*{boundarycomplexmilnorfiber}, recovering a result of Dimca and Lehrer \cite[Theorem 1.3]{DimcaLehrer2012}.

\begin{corollary}
\label{milnorfibertopweightcohomology}
Suppose that $\cM$ has no pairs of distinct parallel elements, let $k = \C$, and let $\cA \in \Gr_\cM(k)$. Then the dimensions of the top weight cohomology of $F_\cA$ are given by
\[
	\dim \Gr^W_{2(d-1)}H^i(F_\cA) = \begin{cases} \mu(\cM), &i = d-1\\ 0, &\text{otherwise} \end{cases}.
\]
\end{corollary}

\begin{acknowledgements}
We would like to acknowledge useful discussions with Dori Bejleri, Daniel Corey, Netanel Friedenberg, Dave Jensen, Kalina Mincheva, Sam Payne, and Dhruv Ranganathan. The second named author was supported by NSF Grant DMS-1702428 and a Graduate Research Fellowship from the NSF.
\end{acknowledgements}

\section{Preliminaries}

We will set notation and recall facts about the equivariant Grothendieck ring of varieties, tropical compactifications, matroids, linear subspaces, and Milnor fibers of hyperplane arrangements.

\subsection{Equivariant Grothendieck ring of varieties}

Let $G$ be a finite group. An action of $G$ on a scheme is said to be \emph{good} if each orbit is contained in an affine open subscheme. For example, any $G$-action on a quasiprojective $k$-scheme is good.

We will recall the definition of the $G$-equivariant Grothendieck ring of varieties $K_0^{G}(\Var_k)$. As a group, $K_0^{G}(\Var_k)$ is generated by symbols $[X, G]$ for each separated finite type $k$-scheme $X$ with good $G$-action, up to $G$-equivariant isomorphism. We then obtain the group $K_0^{G}(\Var_k)$ by imposing the following relations.

\begin{itemize}

\item $[X, G] = [Y, G] + [X \setminus Y, G]$ if $Y$ is a closed $G$-equivariant subscheme of $X$.

\item $[V, G] = [W, G]$ if $V$ and $W$ are $G$-equivariant affine bundles, of the same rank and over the same separated finite type $k$-scheme, with affine $G$-action.

\end{itemize}

The ring structure of $K_0^{G}(\Var_k)$ is given by $[X, G][Y, G] = [X \times_k Y, G]$, where $X \times_k Y$ is endowed with the diagonal $G$-action. We will let $\bL \in K_0^G(\Var_k)$ denote the class of $\bA_k^1$ with trivial $G$-action. For more information on the ring $K_0^G(\Var_k)$, see for example \cite[Definition 4.1]{Hartmann}. 

In the case where $G = \mu_n \subset k^\times$ is the group of $n$th roots of unity, the ring $K_0^{\mu_n}(\Var_k)$ plays an important role in Denef and Loeser's theory of motivic zeta functions and motivic nearby fibers. In particular, the equivariant Grothendieck ring of varieties is used to encode the monodromy action for the Denef-Loeser motivic zeta function, which is related to the monodromy action on the cohomology of the topological Milnor fiber. We refer to \cite{DenefLoeser} for more on these ideas.

\subsection{Tropical compactifications}

Let $n \in \Z_{>0}$, let $M \cong \Z^n$ be a lattice, let $N = M^\vee = \Hom(M, \Z)$, and let $T = \Spec(k[M])$ be the algebraic torus with character lattice $M$.

By a \emph{cone in $N$}, we will mean a rational pointed cone in $N_\R = N \otimes_\Z \R$, and by a \emph{fan in $N$}, we will mean a fan in $N_\R$ consisting of cones in $N$.

Let $d \in \Z_{>0}$, and let $X$ be a pure dimension $d$ reduced closed subscheme of $T$. Let $\Delta$ be a fan in $N$, let $Y(\Delta)$ be the $T$-toric variety defined by $\Delta$, and let $X^\Delta$ be the closure of $X$ in $Y(\Delta)$. Then we may consider the multiplication map $T \times_k X^\Delta \to Y(\Delta)$. The fan $\Delta$ is called a \emph{tropical fan} for $X \hookrightarrow T$ if $X^\Delta$ is complete and the multiplication map $T \times_k X^\Delta \to Y(\Delta)$ is faithfully flat. If $\Delta$ is a tropical fan for $X \hookrightarrow T$, then $X^\Delta$ is called a \emph{tropical compactification}. If $\Delta$ is a tropical fan for $X \hookrightarrow T$ and the multiplication map $T \times_k X^\Delta \to Y(\Delta)$ is smooth, then $X^\Delta$ is called a \emph{sch\"{o}n compactification}. If $\Delta$ is unimodular and $X^\Delta$ is a sch\"{o}n compactification, then $X^\Delta$ is smooth and $X^\Delta \setminus X$ is a simple normal crossing divisor. If there exists some $\Delta$ such that $X^\Delta$ is a sch\"{o}n compactification, then $X$ is said to be \emph{sch\"{o}n} in $T$. It can be shown that $X$ is sch\"{o}n in $T$ if and only if all of its initial degenerations $\init_w X$ are smooth. Also, if $X$ is sch\"{o}n in $T$, then any tropical fan for $X$ gives a sch\"{o}n compactification. We refer to \cite{Tevelev} where tropical compactifications were introduced.

We recall the following elementary fact. See for example \cite[Lemma 3.6]{HelmKatz}.

\begin{fact}
Let $\Delta$ be a tropical fan for $X \hookrightarrow T$, let $\sigma \in \Delta$, and let $\cO_\sigma$ be the torus orbit of $Y(\Delta)$ associated to $\sigma$. Then for all $w \in \relint(\sigma)$, the initial degeneration $\init_w X$ is isomorphic to $\bG_{m,k}^{\dim \sigma} \times_k (X^\Delta \cap \cO_\sigma)$.
\end{fact}

We will also use the following theorem of Luxton and Qu \cite[Theorem 1.5]{LuxtonQu}.

\begin{theorem}[Luxton--Qu]
\label{LuxtonQufansupportedontropistropicalfan}
If $X$ is sch\"{o}n in $T$, then any fan $\Delta$ in $N$ that is supported on $\Trop(X)$ is a tropical fan for $X \hookrightarrow T$.
\end{theorem}

\subsection{Matroids}

Let $d,n \in \Z_{>0}$, and let $\cM$ be a rank $d$ matroid on $\{1, \dots, n\}$. We will let $\cB(\cM)$ denote the set of bases in $\cM$. If $w = (w_1, \dots, w_n) \in \R^n$, we will set
\[
	\cB(\cM_w) = \{B \in \cB(\cM) \, | \, \sum_{i \in B} w_i = \max_{B' \in \cB(\cM)} \sum_{i \in B'} w_i\}.
\]
Then for all $w \in \R^n$, the set $\cB(\cM_w)$ consists of the bases of a rank $d$ matroid on $\{1, \dots, n\}$, and we will let $\cM_w$ denote that matroid. We will let $\Trop(\cM) \subset \R^n$ denote the Bergman fan of $\cM$, so
\[
	\Trop(\cM) = \{w \in \R^n \, | \, \text{$\cM_w$ is loop-free}\}.
\]
If $B \in \cB(\cM)$ and $i \in \{1, \dots, n\} \setminus B$, we will let $C(\cM, i, B)$ denote the fundamental circuit of $B$ with respect to $i$ in $\cM$, i.e., $C(\cM, i, B)$ is the unique circuit of $\cM$ that is contained in $B \cup \{i\}$.

\subsection{Linear subspaces}
\label{preliminarieslinearsubspaces}

Let $d,n \in \Z_{>0}$. We will let $\Gr_{d,n}$ denote the Grassmannian of $d$-dimensional linear subspaces in $\bA_k^n = \Spec(k[x_1, \dots, x_n])$, and for each $\cA \in \Gr_{d,n}(k)$, we will let $X_\cA \subset \bA_k^n$ denote the corresponding linear subspace. For each $w \in \R^n$ and $\cA \in \Gr_{d,n}(k)$, the ideal generated by 
\[
	\{\init_w f \, | \, \text{$f$ in the ideal defining $X_\cA$ in $\bA_k^n$}\} \subset k[x_1, \dots, x_n]
\]
defines a $d$-dimensional linear subspace in $\bA_k^n$, and we let $\cA_w \in \Gr_{d,n}(k)$ denote the point corresponding to that linear subspace.

The combinatorial type of $X_\cA$ is the rank $d$ matroid on $\{1, \dots, n\}$ in which $I \subset \{1, \dots, n\}$ is independent if and only if the set of restrictions $\{x_i|_{X_\cA} \, | \, i \in I\}$ is linearly independent. For each rank $d$ matroid $\cM$ on $\{1, \dots, n\}$, let $\Gr_\cM \subset \Gr_{d,n}$ denote the locus parametrizing linear subspaces with combinatorial type $\cM$. $\Gr_\cM$ is a locally closed subset of $\Gr_{d,n}$. We make no essential use of any particular scheme structure on $\Gr_\cM$, so we may as well consider it with its reduced structure. 

\begin{remark}
By a version of Mn\"{e}v universality, the schemes $\Gr_\cM$ can be arbitrarily singular and have arbitrarily many connected components. For a precise statement, see for example \cite[Proposition 9.7 and Theorem 9.8]{Katz}.
\end{remark}

Let $\cM$ be a rank $d$ matroid on $\{1, \dots, n\}$ and let $w \in \R^n$. For all $\cA \in \Gr_\cM(k)$, we have that $\cA_w \in \Gr_{\cM_w}(k)$. We briefly recall an elementary fact about the map $\Gr_\cM(k) \to \Gr_{\cM_w}(k): \cA \mapsto \cA_w$. Let $\Gr_{d,n} \hookrightarrow \Proj(k[y_I \, | \, I \subset \{1, \dots, n\}, \#I = d])$ be the Pl\"{u}cker embedding, and consider the rational map on the ambient projective space given in homogeneous coordinates by $(a_I)_I \mapsto (a^w_I)_I$, where $a^w_I = a_I$ if $I \in \cB(\cM_w)$ and $a^w_i = 0$ otherwise. It is straightforward to check that this rational map induces the map $\Gr_\cM(k) \to \Gr_{\cM_w}(k): \cA \mapsto \cA_w$. In particular, we get the following facts that will be used in the induction step of our proof of \autoref*{additiveinvariantsmilnorfiber}.

\begin{fact}
Let $\cM$ be a rank $d$ matroid on $\{1, \dots, n\}$, and suppose that $\cA_1, \cA_2 \in \Gr_\cM(k)$ are in the same connected component of $\Gr_\cM$. Then for all $w \in \R^n$, we have that $(\cA_1)_w, (\cA_2)_w \in \Gr_{\cM_w}(k)$ are in the same connected component of $\Gr_{\cM_w}$.
\end{fact}

\begin{fact}
Let $\cM$ be a rank $d$ matroid on $\{1, \dots, n\}$, and suppose that $w \in \R^n$ is such that $\cM_w = \cM$. Then the map $\Gr_\cM(k) \to \Gr_\cM(k): \cA \mapsto \cA_w$ is the identity.
\end{fact}

We now establish notation for certain useful linear forms. If $C$ is a circuit of $\cM$ and $\cA \in \Gr_\cM(k)$, we will let $L_C^\cA \in k[x_1, \dots, x_n]$ denote a linear form in the ideal defining $X_\cA$ in $\bA_k^n$ such that the coefficient of $x_i$ in $L_C^\cA$ is nonzero if and only if $i \in C$. Such an $L_C^\cA$ exists and is unique up to nonzero scalar multiple. Once and for all, we fix such an $L_C^\cA$ for all $C$ and $\cA$.

\subsection{Milnor fibers of hyperplane arrangements}
Let $d, n \in \Z_{>0}$. We will let $\bG_{m,k}^n = \Spec(k[x_1^{\pm 1}, \dots, x_n^{\pm 1}]) \subset \bA_k^n$ denote the complement of the coordinate hyperplanes, and for each $\cA \in \Gr_{d,n}(k)$, we will let $U_\cA$ denote the intersection $X_\cA \cap \bG_{m,k}^n$. In the context of tropical geometry, we will consider each $U_\cA$ as a closed subscheme of $\bG_{m,k}^n$. For all $\cA \in \Gr_{d,n}(k)$ and all $w \in \R^n$, the initial degeneration $\init_w U_\cA$ is equal to $U_{\cA_w}$. For all rank $d$ matroids $\cM$ on $\{1, \dots, n\}$ and all $\cA \in \Gr_{\cM}(k)$, the tropicalization $\Trop(U_\cA) \subset \R^n$ is equal to $\Trop(\cM)$.

For each $\cA \in \Gr_{d,n}(k)$, we will let $F_\cA$ denote the intersection of $X_\cA$ with the closed subscheme of $\bA_k^n$ defined by $(x_1 \cdots x_n -1)$. We will endow each $F_\cA$ with the restriction of the $\mu_n$-action on $\bA_k^n$ where each $\xi \in \mu_n$ acts by scalar multiplication. In the context of tropical geometry, we will consider each $F_\cA$ as a closed subscheme of the algebraic torus $\bG_{m,k}^n$.

Note that if $\cA \in \Gr_{d,n}(k)$ is such that $X_\cA$ is not contained in any coordinate hyperplane of $\bA_k^n$, then the restrictions of the coordinates $x_i$ define a central essential hyperplane arrangement in $X_\cA$. In that case, $U_\cA$ is the complement of that hyperplane arrangement, and $F_\cA$ with its $\mu_n$-action is that arrangement's Milnor fiber with its monodromy action. Note that if $X_\cA$ has combinatorial type $\cM$ in the sense of Section \ref*{preliminarieslinearsubspaces}, then $X_\cA$ is not contained in a coordinate hyperplane if and only if the matroid $\cM$ is loop-free, in which case the corresponding central essential hyperplane arrangement in $X_\cA$ has combinatorial type $\cM$.

If $\cM$ is a rank $d$ loop-free matroid on $\{1, \dots, n\}$ and the characteristic of $k$ does not divide $n$, then $F_\cA$ is smooth and pure dimension $d-1$ for all $\cA \in \Gr_\cM(k)$.

Let $\cM$ be a rank $d$ loop-free matroid on $\{1, \dots, n\}$. The relation of being parallel in $\cM$ is an equivalence relation on $\{1, \dots, n\}$. Let $I_1, \dots, I_m \subset \{1, \dots, n\}$ be the resulting equivalence classes. If the greatest common divisor of $\{ \# I_\ell \, | \, \ell \in \{1, \dots, m\}\}$ is equal to 1, then $F_\cA$ is irreducible for any $\cA \in \Gr_{\cM}(k)$.

\section{Generators for the ideal defining a linear subspace}

Let $d, n \in \Z_{>0}$, and let $\cM$ be a rank $d$ matroid on $\{1, \dots, n\}$. We first recall the following well-known lemma. See \cite[Exercise 1.2.5]{Oxley}.

\begin{lemma}
\label{fundamentalcircuitbasisexchange}
Let $B \in \cB(\cM)$, and let $i \in \{1, \dots, n\} \setminus B$. Then
\[
	C(\cM, i, B) \setminus \{i\} = \{j \in B \, | \, (B \setminus \{j\}) \cup \{ i \} \in \cB(\cM)\}.
\]
\end{lemma}

We now prove some elementary facts about matroids.

\begin{proposition}
\label{wmaximalextrahasminimalweight}
Let $w = (w_1, \dots, w_n) \in \R^n$, let $B \in \cB(\cM_w)$, and let $i \in \{1, \dots, n\} \setminus B$. Then
\[
	\min_{j \in C(\cM, i, B)} w_j = w_i.
\]
\end{proposition}

\begin{proof}
Let $j \in C(\cM, i, B) \setminus \{i\}$. Then by \autoref*{fundamentalcircuitbasisexchange},
\[
	(B \setminus \{j\}) \cup \{i\} \in \cB(\cM).
\]
Because $B \in \cB(\cM_w)$, this implies that $w_i \leq w_j$.
\end{proof}

\begin{proposition}
\label{initialformfundamentalcircuit}
Let $w = (w_1, \dots, w_n) \in \R^n$, let $B \in \cB(\cM_w)$, and let $i \in \{1, \dots, n\} \setminus B$. Then
\[
	C(\cM_w, i, B) = \{ j \in C(\cM, i, B) \, | \, w_j = w_i\}.
\]
\end{proposition}

\begin{proof}
Let $j \in B$. By \autoref*{fundamentalcircuitbasisexchange},
\begin{align*}
	j \in C(\cM_w, i, B) &\iff (B \setminus \{j\}) \cup \{i\} \in \cB(\cM_w)\\
	&\iff \text{$(B \setminus \{j\}) \cup \{i\} \in \cB(\cM)$ and $w_j = w_i$}\\
	&\iff \text{$j \in C(\cM, i , B)$ and $w_j = w_i$,}
\end{align*}
and we are done.
\end{proof}

We will now study certain generators of the ideal defining a linear subspace.

\begin{proposition}
\label{initialformofcicuitformisinitialcircuitform}
Let $\cA \in \Gr_\cM(k)$, let $w \in \R^n$, let $B \in \cB(\cM_w)$, and let $i \in \{1, \dots, n\} \setminus B$. Then $\init_w L_{C(\cM, i, B)}^\cA$ is a nonzero scalar multiple of $L_{C(\cM_w, i, B)}^{\cA_w}$.
\end{proposition}

\begin{proof}
By \autoref*{wmaximalextrahasminimalweight} and \autoref*{initialformfundamentalcircuit}, the coefficient of $x_i$ in the initial form $\init_w L_{C(\cM, i, B)}^\cA$ is nonzero if and only if $i \in C(\cM_w, i, B)$. The proposition then follows from the fact that $\init_w L_{C(\cM, i, B)}^\cA$ is in the ideal defining $X_{\cA_w}$ in $\bA_k^n$.
\end{proof}

\begin{proposition}
\label{fundamentalcircuitsdefinebasisforlinearsubspace}
Let $\cA \in \Gr_\cM(k)$ and $B \in \cB(\cM)$. Then the ideal defining $X_\cA$ in $\bA_k^n$ is generated by
\[
	\{L_{C(\cM, i, B)}^\cA \, | \, i \in \{1, \dots, n\} \setminus B\} \subset k[x_1, \dots, x_n].
\]
\end{proposition}

\begin{proof}
For any $i, j \in \{1, \dots, n\} \setminus B$, we have that $i \in C(\cM, j, B)$ if and only if $i = j$. Therefore $\{L_{C(\cM, i, B)}^\cA \, | \, i \in \{1, \dots, n\} \setminus B\}$ is a set of $n-d$ linearly independent linear forms in the ideal defining $X_\cA$ in $\bA_k^n$, so we are done.
\end{proof}

\begin{proposition}
\label{fundamentalcircuitsgivegrobnerbasis}
Let $\cA \in \Gr_\cM(k)$, let $w \in \R^n$, and let $B \in \cB(\cM_w)$. Then the ideal defining $X_\cA$ in $\bA_k^n$ is generated by
\[
	\{L_{C(\cM,i,B)}^\cA \, | \, i \in \{1, \dots, n\} \setminus B\} \subset k[x_1, \dots, x_n],
\]
and the ideal defining $\init_w U_\cA$ in $\bG_{m,k}^n$ is generated by
\[
	\{\init_w L_{C(\cM, i, B)}^\cA \, | \, i \in \{1, \dots, n\} \setminus B \} \subset k[x_1^{\pm 1}, \dots, x_n^{\pm 1}].
\]
\end{proposition}

\begin{proof}
By \autoref*{fundamentalcircuitsdefinebasisforlinearsubspace}, the ideal defining $X_\cA$ in $\bA_k^n$ is generated by
\[
	\{L_{C(\cM,i,B)}^\cA \, | \, i \in \{1, \dots, n\} \setminus B\} \subset k[x_1, \dots, x_n],
\]
and the ideal defining $X_{\cA_w}$ in $\bA_k^n$ is generated by
\[
	\{L_{C(\cM_w,i,B)}^{\cA_w} \, | \, i \in \{1, \dots, n\} \setminus B\} \subset k[x_1, \dots, x_n].
\]
By \autoref*{initialformofcicuitformisinitialcircuitform}, this implies that the ideal defining $X_{\cA_w}$ in $\bA_k^n$ is generated by
\[
	\{\init_w L_{C(\cM,i,B)}^\cA \, | \, i \in \{1, \dots, n\} \setminus B\} \subset k[x_1, \dots, x_n].
\]
Then we are done by the fact that $\init_w U_\cA = U_{\cA_w} = \bG_{m,k}^n \cap X_{\cA_w}$.
\end{proof}

\section{Initial degenerations of Milnor fibers}

Let $d,n \in \Z_{>0}$ and let $\cM$ be a rank $d$ matroid on $\{1, \dots, n\}$. In this section, we will compute the initial degenerations of the Milnor fiber of a hyperplane arrangement. We will begin by stating the main result of this section.

\begin{theorem}
\label{initialmilnorismilnorinitial}
If $\cA \in \Gr_\cM(k)$ and $w \in \Trop(\cM) \cap (1, \dots, 1)^\perp$, then
\[
	\init_w F_\cA = F_{\cA_w}.
\]
\end{theorem}

Before proving \autoref*{initialmilnorismilnorinitial}, we show that it implies the next two corollaries, which will eventually be used in proving the main results of this paper.

\begin{corollary}
\label{tropicalizationofmilnorfiber}
If $\cA \in \Gr_\cM(k)$, then
\[
	\Trop(F_\cA) = \Trop(\cM) \cap (1, \dots, 1)^\perp.
\]
\end{corollary}

\begin{proof}
Suppose $w \in \R^n \setminus (1, \dots, 1)^\perp$. Then $\init_w (x_1 \cdots x_n - 1)$ is a unit that is in the ideal defining $\init_w F_\cA$ in $\bG_{m,k}^n$, so $\init_w F_\cA = \emptyset$.

Suppose $w \in \R^n \setminus \Trop(\cM)$. Then $\init_w F_\cA \subset \init_w U_\cA = \emptyset$.

Now suppose that $w \in \Trop(\cM) \cap (1, \dots, 1)^\perp$. We only need to show that $\init_w F_{\cA} \neq \emptyset$. By \autoref*{initialmilnorismilnorinitial}, we have that $\init_w F_{\cA} = F_{\cA_w}$. Because $w \in \Trop(\cM)$, we have that $\cM_w$ is loop-free, so $X_{\cA_w}$ is not contained in a coordinate hyperplane of $\bA_k^n$. Thus $F_{\cA_w} \neq \emptyset$.
\end{proof}

We note that if the characteristic of $k$ does not divide $n$, then \autoref*{initialmilnorismilnorinitial} and \autoref*{tropicalizationofmilnorfiber} imply that $F_\cA$ is sch\"{o}n for all $\cA \in \Gr_\cM(k)$.

\begin{corollary}
\label{grobnercompleteintersectionbasismilnorfiber}
Let $\cA \in \Gr_\cM(k)$, let $w \in \R^n$, and let $B \in \cB(\cM_w)$. Then the ideal defining $F_\cA$ in $\bG_{m,k}^n$ is generated by
\[
	\{L_{C(\cM, i, B)}^\cA \, | \, i \in \{1, \dots, n\} \setminus B\} \cup \{x_1 \dots x_n - 1\} \subset k[x_1^{\pm 1}, \dots, x_n^{\pm 1}],
\]
and the ideal defining $\init_w F_\cA$ in $\bG_{m,k}^n$ is generated by
\[
	\{\init_w L_{C(\cM, i, B)}^\cA \, | \, i \in \{1, \dots, n\} \setminus B\} \cup \{\init_w(x_1 \dots x_n - 1)\} \subset k[x_1^{\pm 1}, \dots, x_n^{\pm 1}].
\]
\end{corollary}

\begin{proof}
By \autoref*{fundamentalcircuitsgivegrobnerbasis}, the ideal defining $F_\cA$ in $\bG_{m,k}^n$ is generated by
\[
	\{L_{C(\cM, i, B)}^\cA \, | \, i \in \{1, \dots, n\} \setminus B\} \cup \{x_1 \dots x_n - 1\} \subset k[x_1^{\pm 1}, \dots, x_n^{\pm 1}].
\]
Set $I^w$ to be the ideal in $k[x_1^{\pm 1}, \dots, x_n^{\pm 1}]$ generated by
\[
	\{\init_w L_{C(\cM, i, B)}^\cA \, | \, i \in \{1, \dots, n\} \setminus B\} \cup \{\init_w(x_1 \dots x_n - 1)\}.
\]
Suppose that $w \in \R^n \setminus (1, \dots, 1)^\perp$. Then $\init_w (x_1 \cdots x_n - 1)$ is a unit, so $\init_w F_\cA = \emptyset$ is defined by $I^w$.

Suppose that $w \in \R^n \setminus \Trop(\cM) = \R^n \setminus \Trop(U_\cA)$. Then \autoref*{fundamentalcircuitsgivegrobnerbasis} implies that $\{\init_w L_{C(\cM, i, B)}^\cA \, | \, i \in \{1, \dots, n\} \setminus B\}$ generates the unit ideal in $k[x_1^{\pm 1}, \dots, x_n^{\pm 1}]$, so $\init_w F_\cA = \emptyset$ is defined by $I^w$.

Finally, suppose that $w \in \Trop(\cM) \cap (1, \dots, 1)^\perp$. Then $\init_w (x_1 \cdots x_n - 1) = (x_1 \cdots x_n - 1)$, so \autoref*{fundamentalcircuitsgivegrobnerbasis} implies that $\init_w F_\cA = F_{\cA_w}$ is defined by $I^w$.
\end{proof}

\subsection{Gr\"{o}bner bases for Milnor fibers}

The remainder of this section will be dedicated to proving \autoref*{initialmilnorismilnorinitial}. If $\cM$ has a loop, then $\Trop(\cM) = \emptyset$, so we will assume for the remainder of this section that $\cM$ is loop-free.

For all $I \subset \{1, \dots, n\}$, let $x_I \in k[x_1, \dots, x_n]$ denote the monomial $\prod_{i \in I} x_i$. For any $\cA \in \Gr_{\cM}(k)$, let $I_\cA \subset k[x_1, \dots, x_n]$ denote the ideal defining $F_\cA$ in $\bA_k^n$. For each $\cA \in \Gr_{\cM}(k)$, each circuit $C$ in $\cM$, and each $i \in C$, let $L_{C,i}^\cA \in k[x_1, \dots, x_n]$ denote the nonzero scalar multiple of $L_C^\cA$ whose coefficient of $x_i$ is equal to $1$. For each $\cA \in \Gr_{\cM}(k)$ and each $B \in \cB(\cM)$, let $g_B^\cA \in k[x_1, \dots, x_n]$ denote
\[
	g_B^\cA = x_B \prod_{i \in \{1, \dots, n\} \setminus B} (x_i - L_{C(\cM, i, B), i}^\cA).
\]
By construction and the fact that $\cM$ is loop-free, each $g_B^\cA$ is a homogeneous polynomial of degree $n$ in the variables $\{x_i \, | \, i \in B\}$. 

\begin{lemma}
\label{generatorswithgforMilnorideal}
Let $\cA \in \Gr_{\cM}(k)$ and let $B \in \cB(\cM)$. Then $I_\cA$ is generated by
\[
	\{L_{C(\cM,i,B)}^\cA \, | \, i \in \{1, \dots, n\} \setminus B\} \cup \{g_B^\cA - 1\} \subset k[x_1, \dots, x_n].
\]
\end{lemma}

\begin{proof}
We can rewrite $g_B^\cA$ as
\[
	g_B^\cA = \sum_{I \subset \{1, \dots, n\} \setminus B} (-1)^{\# I} x_{\{1, \dots, n\} \setminus I} \prod_{i \in I} L_{C(\cM, i, B), i}^\cA.
\]
Therefore
\[
	g_B^\cA -1 = (x_1 \cdots x_n - 1) + \sum_{\substack{ I \subset \{1, \dots, n\} \setminus B\\ I \neq \emptyset}}(-1)^{\# I} x_{\{1, \dots, n\} \setminus I} \prod_{i \in I} L_{C(\cM, i, B), i}^\cA.
\]
The lemma thus follows from \autoref*{fundamentalcircuitsdefinebasisforlinearsubspace}.
\end{proof}

For each $\cA \in \Gr_\cM(k)$, let $J_\cA \subset k[x_0, x_1, \dots, x_n]$ denote the homogenization of $I_\cA$. We will compute Gr\"{o}bner bases for each $J_\cA$ at certain monomial orders. 

For each $u \in \Z^{\{0, 1, \dots, n\}}$, let $x^u \in k[x_0, \dots, x_n]$ denote the monomial $(x_0, \dots, x_n)^u$. We will always assume a monomial order $\prec$ satisfies $x^u \preceq 1$, and for all $f \in k[x_0, \dots, x_n]$, we will let the initial term $\init_\prec f$ be the term of $f$ that is $\prec$-minimal.

\begin{remark}
It is more standard in Gr\"{o}bner theory to use the opposite convention, i.e., that 1 is the minimal monomial and that $\init_\prec f$ is the maximal term, but we have chosen our convention to be consistent with our convention for initial forms.
\end{remark}

If $\prec$ is a monomial order on $k[x_0, \dots, x_n]$ and $v \in \R_{\leq 0}^{\{0, \dots, n\}}$, let $\prec_v$ denote the monomial order defined by
\[
	x^{u_1} \prec_v x^{u_2} \quad \iff \quad \text{$u_1 \cdot v < u_2 \cdot v$ or $u_1 \cdot v = u_2 \cdot v$ and $x^{u_1} \prec x^{u_2}$}.
\]
Note that because $v$ has nonpositive entries, $\prec_v$ satisfies $x^u \preceq_v 1$ for all $u$.

\begin{proposition}
\label{grobnertiebreakerhomogeneous}
Let $\cA \in \Gr_\cM(k)$, let $w \in \R^n$, let $B \in \cB(\cM_w)$, and let $\lambda \in \R$ be such that $v = (0,w) + \lambda(1, \dots, 1) \in \R_{\leq 0}^{\{0, \dots, n\}}$. Let $\prec$ be a monomial order on $k[x_0, \dots, x_n]$ such that $x_i \prec x_j$ for all $i \in \{1, \dots, n\} \setminus B$ and $j \in B$. Then
\[
	\{L_{C(\cM,i,B)}^\cA \, | \, i \in \{1, \dots, n\} \setminus B\} \cup \{g_B^\cA - x_0^n\} \in k[x_0, \dots, x_n]
\]
is a Gr\"{o}bner basis for $J_\cA$ with respect to $\prec_v$.
\end{proposition}

\begin{proof}
By \autoref*{generatorswithgforMilnorideal}, $\{L_{C(\cM,i,B)}^\cA \, | \, i \in \{1, \dots, n\} \setminus B\} \cup \{g_B^\cA - x_0^n\}$ generates $J_\cA$. 

Let $i \in \{1, \dots, n\} \setminus B$. We will show that $\init_{\prec_v} L_{C(\cM, i, B)}^\cA$ is a nonzero scalar multiple of $x_i$. By \autoref*{wmaximalextrahasminimalweight} and \autoref*{initialformfundamentalcircuit}, the coefficient of $x_j$ in the initial form $\init_w L_{C(\cM, i, B)}^\cA$ is nonzero if and only if $j \in C(\cM_w, i, B)$. Because $L_{C(\cM, i, B)}^\cA$ is homogeneous, this implies that the coefficient of $x_j$ in the initial form $\init_v L_{C(\cM, i, B)}^\cA$ is nonzero if and only if $j \in C(\cM_w, i, B)$. Because $C(\cM_w, i, B) \subset B \cup \{i\}$, this implies that $\init_{\prec_v} L_{C(\cM, i, B)}^\cA$ is a nonzero scalar multiple of $x_i$.

Therefore, because $g_B^\cA - x_0^n$ is a polynomial in $\{x_0\} \cup \{x_j \, | \, j \in B\}$, the initial terms $\{\init_{\prec_v}L_{C(\cM,i,B)}^\cA \, | \, i \in \{1, \dots, n\} \setminus B\} \cup \{\init_{\prec_v}(g_B^\cA - x_0^n)\}$ are pairwise relatively prime. The proposition then follows from Buchberger's criteria.
\end{proof}

We can now apply this result to better understand each $I_\cA$.

\begin{proposition}
\label{generatorsforMilnorfiberinaffinespace}
Let $\cA \in \Gr_\cM(k)$, let $w \in \R^n$, and let $B \in \cB(\cM_w)$. Then
\[
	\{\init_w L_{C(\cM,i,B)}^\cA \, | \, i \in \{1, \dots, n\} \setminus B\} \cup \{\init_w(g_B^\cA - 1)\}
\]
generates the ideal $(\init_w f \, | \, f \in I_\cA) \subset k[x_1, \dots, x_n]$.
\end{proposition}

\begin{proof}
Let $\lambda \in \R$ be such that $v = (0,w) + \lambda(1,\dots, 1) \in \R_{\leq 0}^{\{0, \dots, n\}}$, and let $J_{\cA}^w \subset k[x_0, \dots, x_n]$ be the ideal
\[
	(\init_{(0,w)} f \, | \, f \in J_\cA) =  (\init_{v} f \, | \, f \in J_\cA).
\]
We will now use a standard argument to show that \autoref*{grobnertiebreakerhomogeneous} implies that $J_\cA^w$ is generated by
\[
	\{\init_{(0,w)} L_{C(\cM,i,B)}^\cA \, | \, i \in \{1, \dots, n\} \setminus B\} \cup \{\init_{(0,w)}(g_B^\cA - x_0^n)\}.
\]
Let $\prec$ be a monomial order on $k[x_0, \dots, x_n]$ such that $x_i \prec x_j$ for all $i \in \{1, \dots, n\} \setminus B$ and $j \in B$.
Let $\init_{\prec_v} J_\cA = \init_{\prec}J_\cA^w \subset k[x_0, \dots, x_n]$ be the ideal
\[
	(\init_{\prec_v} f \, | \, f \in J_\cA) = (\init_{\prec} f \, | \, f \in J_{\cA}^w).
\]
Then by \autoref*{grobnertiebreakerhomogeneous}, the initial ideal $\init_{\prec}J_\cA^w = \init_{\prec_v}J_\cA$ is generated by
\[
	\{\init_{\prec}(\init_{(0,w)} L_{C(\cM,i,B)}^\cA) \, | \, i \in \{1, \dots, n\} \setminus B\} \cup \{\init_{\prec}(\init_{(0,w)}(g_B^\cA - x_0^n))\}.
\]
This implies that $J_\cA^w$ is generated by
\[
	\{\init_{(0,w)} L_{C(\cM,i,B)}^\cA \, | \, i \in \{1, \dots, n\} \setminus B\} \cup \{\init_{(0,w)}(g_B^\cA - x_0^n)\}.
\]
Now let $f \in I_\cA$, and let $g \in J_\cA$ be the homogenization of $f$. Then $\init_{(0,w)} g \in J_\cA^w$ is in the ideal of $k[x_0, \dots, x_n]$ generated by
\[
	\{\init_{(0,w)} L_{C(\cM,i,B)}^\cA \, | \, i \in \{1, \dots, n\} \setminus B\} \cup \{\init_{(0,w)}(g_B^\cA - x_0^n)\},
\]
so $\init_w f = (\init_{(0,w)} g)|_{x_0 =1}$ is in the ideal of $k[x_1, \dots, x_n]$ generated by
\[
	\{\init_w L_{C(\cM,i,B)}^\cA \, | \, i \in \{1, \dots, n\} \setminus B\} \cup \{\init_w(g_B^\cA - 1)\},
\]
completing our proof.
\end{proof}

In the next two lemmas, we compute certain initial forms of each $(g_B^\cA-1)$.

\begin{lemma}
\label{initialformxicancelnormalizedcircuitforms}
Let $\cA \in \Gr_{\cM}(k)$, let $w = (w_1, \dots, w_n) \in \Trop(\cM)$, let $B \in \cB(\cM_w)$, and let $i \in \{1, \dots, n\} \setminus B$. Then
\[
	\init_w(x_i - L_{C(\cM,i,B), i}^\cA) = x_i - L_{C(\cM_w, i, B), i}^{\cA_w},
\]
and for all $u \in \supp( x_i - L_{C(\cM_w, i, B), i}^{\cA_w})$,
\[
	u \cdot w = w_i.
\]
\end{lemma}

\begin{proof}
By \autoref*{wmaximalextrahasminimalweight} and \autoref*{initialformfundamentalcircuit},
\[
	\init_w L_{C(\cM,i,B), i}^\cA = L_{C(\cM_w, i, B), i}^{\cA_w}.
\]
Because $\cM_w$ is loop-free, $L_{C(\cM_w, i, B), i}^{\cA_w}$ is not a monomial, so
\[
	\init_w(x_i - L_{C(\cM,i,B), i}^\cA) = x_i - \init_w L_{C(\cM,i,B), i}^\cA = x_i - L_{C(\cM_w, i, B), i}^{\cA_w},
\]
and for all $u \in \supp( x_i - L_{C(\cM_w, i, B), i}^{\cA_w})$,
\[
	u \cdot w = w_i.
\]
\end{proof}

\begin{lemma}
\label{initialdegenerationofextragpolynomial}
Let $\cA \in \Gr_\cM(k)$, let $w \in \Trop(\cM) \cap (1, \dots, 1)^\perp$, and let $B \in \cB(\cM_w)$. Then
\[
	\init_w (g_B^{\cA} -1) = g_B^{\cA_w}-1.
\]
\end{lemma}

\begin{proof}
Write $w = (w_1, \dots, w_n)$. By \autoref*{initialformxicancelnormalizedcircuitforms},
\begin{align*}
	\init_w g_B^\cA &= x_B \prod_{i \in \{1, \dots, n\} \setminus B} \init_w(x_i - L_{C(\cM, i, B), i}^\cA)\\
	&= x_B \prod_{i \in \{1, \dots, n\} \setminus B} (x_i - L_{C(\cM_w, i, B), i}^{\cA_w})\\
	&= g_B^{\cA_w},
\end{align*}
and for all $u \in \supp(\init_w g_B^\cA)$,
\[
	u \cdot w = w_1 + \dots + w_n = 0.
\]
Because $g_B^\cA$ is homogeneous of degree $n$, and in particular $0 \notin \supp(\init_w g_B^\cA)$, this implies that
\[
	\init_w(g_B^\cA - 1) = \init_w(g_B^\cA) - 1 = g_B^{\cA_w} - 1.
\]
\end{proof}

Now let $\cA \in \Gr_\cM(k)$ and $w \in \Trop(\cM) \cap (1, \dots, 1)^\perp$. We complete the proof of \autoref*{initialmilnorismilnorinitial}.

\begin{proof}[Proof of \autoref*{initialmilnorismilnorinitial}]
By construction, the ideal defining $F_\cA$ in $\bG_{m,k}^n$ is generated by the image of $I_\cA$ in $k[x_1^{\pm 1}, \dots, x_n^{\pm 1}]$. Thus the ideal defining $\init_w F_\cA$ in $\bG_{m,k}^n$ is generated by the image of $(\init_w f \, | \, f \in I_\cA) \subset k[x_1, \dots, x_n]$ in $k[x_1^{\pm 1}, \dots, x_n^{\pm 1}]$. 

Let $B \in \cB(\cM_w)$. Then by \autoref*{wmaximalextrahasminimalweight}, \autoref*{initialformfundamentalcircuit}, \autoref*{generatorsforMilnorfiberinaffinespace}, and \autoref*{initialdegenerationofextragpolynomial}, the ideal defining $\init_w F_\cA$ in $\bG_{m,k}^n$ is generated by
\[
	\{L_{C(\cM_w,i,B)}^{\cA_w} \, | \, i \in \{1, \dots, n\} \setminus B\} \cup \{g_B^{\cA_w} - 1\} \subset k[x_1^{\pm 1}, \dots, x_n^{\pm 1}].
\]
The theorem thus follows from \autoref*{generatorswithgforMilnorideal}.
\end{proof}

\section{Boundary complex of the Milnor fiber}
\label{boundarycomplexsection}

Let $d, n \in \Z_{>0}$, and let $\cM$ be a rank $d$ loop-free matroid on $\{1, \dots, n\}$. In this section, we will prove \autoref*{boundarycomplexmilnorfiber}. 

We begin by proving the following combinatorial lemma.

\begin{lemma}
\label{parallelafterdegenerationgivesparallelbefore}
Let $w = (w_1, \dots, w_n) \in \Trop(\cM)$, let $i \in \{1, \dots, n\}$ such that $w_i = \max_{\ell \in \{1, \dots, n\}} w_\ell$, and suppose $j \in \{1, \dots, n\}$ is such that $i$ and $j$ are parallel in $\cM_w$. Then $i$ and $j$ are parallel in $\cM$.
\end{lemma}

\begin{proof}
We may assume that $i \neq j$. Because $\cM_w$ is loop-free, there exists $B \in \cB(\cM_w)$ such that $j \in B$. Because $i$ and $j$ are parallel in $\cM_w$, we have that $i \notin B$ and $C(\cM_w, i, B) = \{i,j\}$. By \autoref*{wmaximalextrahasminimalweight} and the hypotheses,
\[
	w_i = \min_{\ell \in C(\cM, i, B)} w_\ell \leq \max_{\ell \in C(\cM, i, B)}w_\ell \leq w_i.
\]
Then by \autoref*{initialformfundamentalcircuit},
\[
	\{i, j\} = C(\cM_w, i, B) = \{\ell \in C(\cM, i, B) \, | \, w_\ell = w_i\} = C(\cM, i, B),
\]
so $i$ and $j$ are parallel in $\cM$.
\end{proof}

We can now prove the following irreducibility statement.

\begin{proposition}
Suppose that $\cM$ has no pairs of distinct parallel elements, and let $\cA \in \Gr_\cM(k)$. Then for all $w \in \R^n$, the initial degeneration $\init_w F_\cA$ is irreducible.
\end{proposition}

\begin{proof}
We may assume that $w \in \Trop(F_\cA) = \Trop(\cM) \cap (1, \dots, 1)^\perp$. Then by \autoref*{parallelafterdegenerationgivesparallelbefore}, there exists $i \in \{1, \dots, n\}$ that is not parallel in $\cM_w$ to any other element. Therefore $F_{\cA_w}$, which is equal to $\init_w F_\cA$ by \autoref*{initialmilnorismilnorinitial}, is irreducible.
\end{proof}

Now suppose that $\cM$ has no pairs of distinct parallel elements, assume that the characteristic of $k$ does not divide $n$, and let $\cA \in \Gr_\cM(k)$. Let $\Delta$ be a unimodular fan in $\Z^n$ supported on $\Trop(F_\cA)$, let $Y(\Delta)$ be the smooth $\bG_{m,k}^n$-toric variety defined by $\Delta$, and let $F_\cA^\Delta$ be the closure of $F_\cA$ in $Y(\Delta)$. Because $F_\cA$ is sch\"{o}n in $\bG_{m,k}^n$, we have that $F_\cA \subset F_\cA^\Delta$ is a simple normal crossing compactification. By \cite{Thuillier}, the dual complex of $F_\cA^\Delta \setminus F_\cA$ is homotopy equivalent to the dual complex of $\overline{F_\cA}\setminus F_\cA$ for any simple normal crossing compactification $F_\cA \subset \overline{F_\cA}$. Therefore \autoref*{boundarycomplexmilnorfiber} follows from the next proposition.

\begin{proposition}
The dual complex of $F_\cA^\Delta \setminus F_\cA$ is homotopy equivalent to a $\mu(\cM)$-wedge of dimension $d-2$ spheres.
\end{proposition}

\begin{proof}
Because $F_\cA$ has irreducible initial degenerations, each intersection of $F_\cA^\Delta$ with a torus orbit of $Y(\Delta)$ is irreducible. Thus the dual complex of $F_\cA^\Delta \setminus F_\cA$ is homeomorphic to the dual complex of $Y(\Delta) \setminus \bG_{m,k}^n$, which is homeomorphic to the intersection of $\Trop(F_\cA) = \Trop(\cM) \cap (1, \dots, 1)^\perp$ with the unit sphere in $\R^n$. By \cite{ArdilaKlivans}[corollary to Theorem 1], this intersection is homotopy equivalent to a $\mu(\cM)$-wedge of dimension $d-2$ spheres.
\end{proof}

\section{Functions defining tropical compactifications}
\label{functionsdefiningtropicalcompactifications}

Throughout Sections \ref*{functionsdefiningtropicalcompactifications}, \ref*{tropicalcompactificationsandgroupactions}, and \ref*{schoncompactificationsinfamilies}, let $n \in \Z_{> 0}$, let $M \cong \Z^n$ be a lattice, let $N = M^\vee$, let $T = \Spec(k[M])$ be the algebraic torus with character lattice $M$, and for each $u \in M$, let $\chi^u \in k[M]$ denote the corresponding character. The following notation and lemma will be used in Sections \ref*{functionsdefiningtropicalcompactifications} and \ref*{tropicalcompactificationsandgroupactions}.

\begin{notation}
If $\sigma$ is a cone in $N$, we will let $\psi_\sigma: k[\sigma^\vee \cap M] \to k[\sigma^\perp \cap M]$ denote the $k$-algebra map defined by
\[
	\chi^u \mapsto \begin{cases}\chi^u, & u \in \sigma^\perp \cap M, \\ 0, &\text{otherwise},  \end{cases}
\]
and we will let $\varphi_\sigma: k[\sigma^\vee \cap M] \to k[M]$ denote the composition of $\psi_\sigma$ with the inclusion $k[\sigma^\perp \cap M] \hookrightarrow k[M]$.
\end{notation}

\begin{lemma}
\label{boundaryinitial}
Let $f \in k[M]$, and let $\sigma$ be a cone in $N$ such that for each face $\tau$ of $\sigma$ and each pair $w_1, w_2 \in \relint(\tau)$, we have that
\[
	\init_{w_1}f = \init_{w_2}f.
\]
Let $u \in M$ be such that $-u \in \supp(\init_w f)$ for all $w \in \relint(\sigma)$.

Then for all faces $\tau$ of $\sigma$,
\[
	\chi^u f \in k[\tau^\vee \cap M],
\]
and for all $w \in \relint(\tau)$,
\[
	\varphi_\tau(\chi^u f) = \chi^u \init_w f.
\]
\end{lemma}

\begin{proof}
By the hypotheses, for each face $\tau$ of $\sigma$, there exists $f_\tau \in k[M]$ such that $f_\tau = \init_w f$ for all $w \in \relint(\tau)$.

Let $\tau$ be a face of $\sigma$. By continuity and the definition of initial forms,
\[
	-u \in \supp(f_\sigma) \subset \supp(f_\tau).
\]
Thus for all $v \in \supp(f_\tau)$ and all $w \in \relint(\tau)$,
\[
	\langle v, w \rangle = \langle -u, w \rangle,
\]
and for all $v \in \supp(f) \setminus \supp(f_\tau)$ and all $w \in \relint(\tau)$,
\[
	\langle v, w \rangle > \langle -u, w \rangle.
\]
Thus by continuity, for all $v \in \supp(f_\tau)$ and all $w \in \tau$,
\[
	\langle v, w \rangle = \langle -u, w \rangle,
\]
and for all $v \in \supp(f) \setminus \supp(f_\tau)$ and all $w \in \tau$,
\[
	\langle v, w \rangle \geq \langle -u, w \rangle.
\]
This implies that for each $v \in \supp(f_\tau)$,
\[
	u+v \in \tau^\perp \cap M,
\]
and for each $v \in \supp(f) \setminus \supp(f_\tau)$,
\[
	u+v \in (\tau^\vee \cap M) \setminus (\tau^\perp \cap M).
\]
Therefore,
\[
	\chi^u f \in k[\tau^\vee \cap M].
\]
Now write
\[
	f = \sum_{v \in \supp(f)} a_v \chi^v,
\]
where each $a_v \in k^\times$. Then for all $w \in \relint(\tau)$,
\begin{align*}
	\varphi_\tau(\chi^u f) &= \sum_{v \in \supp(f_\tau)} a_v \varphi_\tau(\chi^{u+v}) + \sum_{v \in \supp(f) \setminus \supp(f_\tau)} a_v \varphi_\tau(\chi^{u+v})\\
	&= \sum_{v \in \supp(f_\tau)} a_v \chi^{u+v}\\
	&= \chi^u \init_w f.
\end{align*}
\end{proof}

In the remainder of this section we show that under certain conditions, we can give functions that define partial compactifications of subvarieties of tori.

Let $d \in \Z_{>0}$, let $f_1, \dots, f_{n-d} \in k[M]$, let $\Delta, \overline{\Delta}$ be fans in $N$, and let $\sigma \in \overline{\Delta}$. Let $X \hookrightarrow T$ be the closed subscheme defined by the ideal generated by $f_1, \dots, f_{n-d}$, and suppose that the following hypotheses hold.

\begin{enumerate}[(i)]

\item $X$ is pure dimension $d$ and reduced.

\item $\Delta \subset \overline{\Delta}$.

\item $\overline{\Delta}$ is unimodular.

\item $\Delta$ is a tropical fan for $X \hookrightarrow T$.

\item For each face $\tau$ of $\sigma$, each pair $w_1, w_2 \in \relint(\tau)$, and each $i \in \{1, \dots, n-d\}$,
\[
	\init_{w_1}f_i = \init_{w_2} f_i.
\]

\item For each $w \in \sigma$, we have that $\init_w f_1, \dots, \init_w f_{n-d}$ generate the ideal defining $\init_w X$ in $T$.

\end{enumerate}

Let $Y(\sigma) = \Spec(k[\sigma^\vee \cap M])$ be the affine $T$-toric variety defined by $\sigma$, and let $X^\sigma \hookrightarrow Y(\sigma)$ be the closure of $X$.

\begin{proposition}
\label{tcideal}
Let $u_1, \dots, u_{n-d} \in M$ be such that for each $i \in \{1, \dots, n-d\}$,
\[
	-u_i \in \supp(\init_w f_i)
\]
for all, or equivalently some, $w \in \relint(\sigma)$.

Then $\chi^{u_1}f_1, \dots, \chi^{u_{n-d}}f_{n-d} \in k[\sigma^\vee \cap M]$, and $\chi^{u_1}f_1, \dots, \chi^{u_{n-d}}f_{n-d}$ generate the ideal defining $X^\sigma$ in $Y(\sigma)$.
\end{proposition}

\subsection{Proof of \autoref*{tcideal}}

Let $u_1, \dots, u_{n-d} \in M$ be such that for each $i \in \{1, \dots, n-d\}$,
\[
	-u_i \in \supp(\init_w f_i)
\]
for all $w \in \relint(\sigma)$. We now prove the first part of \autoref*{tcideal}.

\begin{proposition}
We have that $\chi^{u_1}f_1, \dots, \chi^{u_{n-d}}f_{n-d} \in k[\sigma^\vee \cap M]$.
\end{proposition}

\begin{proof}
This follows directly from \autoref*{boundaryinitial}.
\end{proof}

Now let $X' \hookrightarrow Y(\sigma)$ be the closed subscheme defined by the ideal generated by $\chi^{u_1}f_1, \dots, \chi^{u_{n-d}}f_{n-d}$. Note that by construction, $X^\sigma$ is a closed subscheme of $X'$.

\begin{proposition}
\label{surjectiveoneachorbit}
Let $\tau$ be a face of $\sigma$, and let $\cO_\tau$ be the torus orbit of $Y(\sigma)$ associated to $\tau$. Then the closed immersion
\[
	X^\sigma \cap \cO_\tau \hookrightarrow X' \cap \cO_\tau
\]
is surjective.
\end{proposition}

\begin{proof}
Let $w \in \relint(\tau)$. Because $\sigma \in \overline{\Delta}$ and $\overline{\Delta}$ contains a tropical fan for $X \hookrightarrow T$, we have that
\[
	\bG_{m,k}^{\dim \tau} \times_k (X^\sigma \cap \cO_\tau) \cong \init_w X.
\]
Let $Y(\tau) = \Spec(k[\tau^\vee \cap M]) \subset Y(\sigma)$ be the affine $T$-toric variety defined by $\tau$. Note that $\cO_\tau = \Spec(k[\tau^\perp \cap M])$, and the inclusion
\[
	\cO_\tau \hookrightarrow Y(\tau)
\]
is given by the map $\psi_\tau: k[\tau^\vee \cap M] \to k[\tau^\perp \cap M]$. Thus by construction, $X' \cap \cO_\tau$ is the closed subscheme of $\cO_\tau$ defined by the ideal in $k[\tau^\perp \cap M]$ generated by
\[
	\psi_\tau(\chi^{u_1}f_1), \dots, \psi_\tau(\chi^{u_{n-d}}f_{n-d}).
\]
Let $T \to \cO_\tau$ be the morphism defined by the inclusion $k[\tau^\perp \cap M] \hookrightarrow k[M]$. Then the preimage of $X' \cap \cO_\tau$ under the morphism $T \to \cO_\tau$ is defined by the ideal in $k[M]$ generated by
\[
	\varphi_\tau(\chi^{u_1}f_1), \dots, \varphi_\tau(\chi^{u_{n-d}}f_{n-d}).
\]
By \autoref*{boundaryinitial}, for each $i \in \{1, \dots, n-d\}$,
\[
	\varphi_\tau(\chi^{u_i}f_i) = \chi^{u_i}\init_w f_i.
\]
Then because $\init_w f_1, \dots, \init_w f_{n-d}$ generate the ideal defining $\init_w X$ in $T$, we have that the preimage of $X' \cap \cO_\tau$ under the morphism $T \to \cO_\tau$ is equal to $\init_w X$. Because the morphism $T \to \cO_\tau$ is isomorphic to the projection morphism $\bG_{m,k}^{\dim \tau} \times_k \cO_\tau \to \cO_\tau$, this implies
\[
	\bG_{m,k}^{\dim \tau} \times_k (X' \cap \cO_\tau) \cong \init_w X \cong \bG_{m,k}^{\dim \tau} \times_k (X^\sigma \cap \cO_\tau).
\]
Therefore there exists a dimension preserving bijection between the irreducible components of $X' \cap \cO_\tau$ and the irreducible components of $X^\sigma \cap \cO_\tau$, and this implies that that the closed immersion $X^\sigma \cap \cO_\tau \hookrightarrow X' \cap \cO_\tau$ is surjective.
\end{proof}

We now get the following corollary.

\begin{corollary}
\label{definingfunctionscandidatesurjective}
The closed immersion $X^\sigma \hookrightarrow X'$ is surjective.
\end{corollary}

\begin{proof}
This is a direct consequence of \autoref*{surjectiveoneachorbit} and the fact that the torus orbits stratify $Y(\sigma)$.
\end{proof}

We can now prove the following.

\begin{proposition}
\label{definingfunctionscandidatereduced}
The scheme $X'$ is reduced.
\end{proposition}

\begin{proof}
Because $X$ has pure dimension $d$, we have that $X^\sigma$ has pure dimension $d$. Thus by \autoref*{definingfunctionscandidatesurjective}, $X'$ has pure dimension $d$. Because $X'$ is defined in $Y(\sigma)$ by $n-d$ functions, this implies that $X'$ is Cohen-Macaulay. By construction, $X' \cap T = X$, and by \autoref*{definingfunctionscandidatesurjective}, $X$ is dense in $X'$. Then because $X$ is reduced, $X'$ is generically reduced and therefore reduced.
\end{proof}

The following corollary completes the proof of \autoref*{tcideal}.

\begin{corollary}
The closed immersion $X^\sigma \hookrightarrow X'$ is an isomorphism.
\end{corollary}

\begin{proof}
This follows directly from \autoref*{definingfunctionscandidatesurjective} and \autoref*{definingfunctionscandidatereduced}.
\end{proof}

\section{Tropical compactifications and group actions}
\label{tropicalcompactificationsandgroupactions}

Let $G$ be a finite group. In this section we will prove that if we equip certain tropical compactifications with certain $G$-actions, then their classes in the $G$-equivariant Grothendieck ring of varieties can be related to classes of initial degenerations.

Let $G \to T$ be an algebraic group homomorphism, where $G$ is considered as a $k$-scheme in the standard way, and endow $T$ with the $G$-action induced by $G \to T$ and the action of $T$ on itself given by left multiplication. 

We begin with the following elementary lemma.

\begin{lemma}
\label{torusequivariantclasstrivial}
We have that
\[
	[T, G] = (\bL - 1)^n \in K_0^G(\Var_k).
\]
\end{lemma}

\begin{proof}
First consider the case where $n=1$. Then the action of $T$ on itself extends to a linear action on $\bA_k^1 \supset T$, so the $G$-action on $T$ extends to a linear action on $\bA_k^1$. Thus
\[
	[T, G] = [\bA_k^1, G] - 1 = \bL - 1 \in K_0^G(\Var_k).
\]
Now consider the case where $n$ is arbitrary, and fix an isomorphism $T \cong \bG_{m,k}^n$. On each $\bG_{m,k}$ factor of $\bG_{m,k}^n$, we have a $G$-action on $\bG_{m,k}$ induced by the composition $G \to T \to \bG_{m,k}$, where the second map is the projection onto the given factor. Then the action of $G$ on $T$ is the diagonal action induced by these actions. Thus by the case above
\[
	[T, G] = (\bL -1)^n \in K_0^G(\Var_k).
\]
\end{proof}

Let $d \in \Z_{>0}$, let $f_1, \dots, f_m \in k[M]$, and let $\Delta$ be a fan in $N$. Let $X \hookrightarrow T$ be the closed subscheme defined by the ideal generated by $f_1, \dots, f_m$, and suppose that the following hypotheses hold.

\begin{enumerate}[(i)]

\item $X$ is pure dimension $d$ and reduced.

\item $\Delta$ is unimodular.

\item $\Delta$ is a tropical fan for $X \hookrightarrow T$.

\item For each $\sigma \in \Delta$, each pair $w_1, w_2 \in \relint(\sigma)$, and each $i \in \{1, \dots, m\}$,
\[
	\init_{w_1} f_i = \init_{w_2} f_i.
\]

\item For each $\sigma \in \Delta$, there exist $i_1, \dots i_{n-d} \in \{1, \dots, m\}$ such that $f_{i_1}, \dots, f_{i_{n-d}}$ generate the ideal defining $X$ in $T$ and $\init_w f_{i_1}, \dots, \init_w f_{i_{n-d}}$ generate the ideal defining $\init_w X$ in $T$ for all $w \in \sigma$.

\end{enumerate}

Let $Y(\Delta)$ be the $T$-toric variety defined by $\Delta$, and endow $Y(\Delta)$ with the $G$-action induced by $G \to T$ and the action of $T$ on $Y(\Delta)$. For each $\sigma \in \Delta$, choose some $w_\sigma \in \relint(\sigma)$. Let $X^\Delta \hookrightarrow Y(\Delta)$ be the closure of $X$. 

\begin{proposition}
\label{tropcompequivariantclass}
Suppose that $X$ is invariant under the $G$-action on $T$. Then
\begin{enumerate}[(a)]

\item $X^\Delta$ is invariant under the $G$-action on $Y(\Delta)$,

\item $\init_{w_\sigma} X$ is invariant under the $G$-action on $T$ for each $\sigma \in \Delta$,

\item and

\[
	(\bL-1)^{\dim \Delta} [X^\Delta, G] = \sum_{\sigma \in \Delta} (\bL-1)^{\dim \Delta - \dim \sigma}[\init_{w_\sigma} X, G] \in K_0^G(\Var_k),
\]
where $X^\Delta$ and each $\init_{w_\sigma} X$ are endowed with the $G$-actions induced by restriction of the $G$-actions on $Y(\Delta)$ and $T$, respectively.

\end{enumerate}
\end{proposition}

\begin{remark}
Note that each $G$-orbit of $Y(\Delta)$ is contained in a $T$-invariant open affine, so the $G$-action on $Y(\Delta)$ is good. The $G$-action on $T$ is good because $T$ is affine. Therefore the classes in the statement of \autoref*{tropcompequivariantclass} are well defined.
\end{remark}

\subsection{Proof of \autoref*{tropcompequivariantclass}}

Suppose that $X$ is invariant under the $G$-action on $T$. We begin by proving the first part of \autoref*{tropcompequivariantclass}.

\begin{proposition}
\label{XDeltainvariant}
We have that $X^\Delta$ is invariant under the $G$-action on $Y(\Delta)$.
\end{proposition}

\begin{proof}
Because $X^\Delta$ is the closure of $X$, as a set $X^\Delta$ is invariant under the $G$-action on $Y(\Delta)$. Because $X$ is reduced, so is $X^\Delta$, and therefore as a closed subscheme $X^\Delta$ is invariant under the $G$-action on $Y(\Delta)$.
\end{proof}

For each $\sigma \in \Delta$, let $\cO_\sigma$ be torus orbit of $Y(\Delta)$ corresponding to $\sigma$.

\begin{proposition}
\label{XDeltastratuminvariant}
For each $\sigma \in \Delta$, the scheme theoretic intersection $X^\Delta \cap \cO_\sigma$ is invariant under the $G$-action on $Y(\Delta)$.
\end{proposition}

\begin{proof}
This follows from \autoref*{XDeltainvariant} and the fact that each torus orbit of $Y(\Delta)$ is invariant under the $G$-action.
\end{proof}

For each $\sigma \in \Delta$, let $T \to \cO_\sigma = \Spec(k[\sigma^\perp \cap M])$ be the algebraic group homomorphism given by the inclusion $k[\sigma^\perp \cap M] \to k[M]$, let $T_\sigma$ be the kernel of $T \to \cO_\sigma$, and fix an identification of $T$ with $T_\sigma \times_k \cO_\sigma$ given by a splitting of the sequence $0 \to T_\sigma \to T \to \cO_\sigma \to 0$.

\begin{remark}
\label{twoactionsfororbitsarethesame}
Note that the $G$-action on $\cO_\sigma$ induced by restriction of the $G$-action on $Y(\Delta)$ is the same as the $G$-action induced by $G \to T \to \cO_\sigma$ and the action of $\cO_\sigma$ on itself given by left multiplication.
\end{remark}

\begin{proposition}
\label{writinginitialformasproduct}
Let $\sigma \in \Delta$. Then under the identification of $T$ with $T_\sigma \times_k \cO_\sigma$, the initial degeneration $\init_{w_\sigma} X$ is equal to $T_\sigma \times_k (X^\Delta \cap \cO_\sigma)$.
\end{proposition}

\begin{proof}
Let $i_1, \dots, i_{n-d} \in \{1, \dots, m\}$ be such that $f_{i_1}, \dots, f_{i_{n-d}}$ generate the ideal defining $X$ in $T$ and $\init_w f_{i_1}, \dots, \init_w f_{i_{n-d}}$ generate the ideal defining $\init_w X$ in $T$ for all $w \in \sigma$. Because $X$ has pure dimension $d$, we have that $f_{i_j} \neq 0$ for all $j \in \{1, \dots, n-d\}$. Thus we may let $u_1, \dots, u_{n-d} \in M$ be such that
\[
	-u_j \in \supp(\init_w f_{i_j})
\]
for all $j \in \{1, \dots, n-d\}$ and all $w \in \relint(\sigma)$. Then by \autoref*{tcideal},
\[
	\chi^{u_1} f_{i_1}, \dots, \chi^{u_{n-d}} f_{i_{n-d}} \in k[\sigma^\vee \cap M],
\]
and $\chi^{u_1} f_{i_1}, \dots, \chi^{u_{n-d}} f_{i_{n-d}}$ generate the ideal defining $X^\sigma$ in $Y(\sigma)$, where $Y(\sigma) = \Spec(k[\sigma^\vee \cap M])$ is the affine $T$-toric variety defined by $\sigma$ and $X^\sigma \hookrightarrow Y(\sigma)$ is the closure of $X$. Thus $X^\Delta \cap \cO_\sigma$ is the closed subscheme of $\cO_\sigma$ defined by the ideal generated by
\[
	\psi_\sigma(\chi^{u_1} f_{i_1}), \dots, \psi_\sigma(\chi^{u_{n-d}} f_{i_{n-d}}) \in k[\sigma^\perp \cap M].
\]
Thus the preimage of $X^\Delta \cap \cO_\sigma$ under the morphism $T \to \cO_\sigma$ is the closed subscheme defined by the ideal generated by
\[
	\varphi_\sigma(\chi^{u_1} f_{i_1}), \dots, \varphi_\sigma(\chi^{u_{n-d}} f_{i_{n-d}}) \in k[M].
\]
Then by \autoref*{boundaryinitial}, the preimage of $X^\Delta \cap \cO_\sigma$ under the morphism $T \to \cO_\sigma$ is the closed subscheme defined by the ideal generated by
\[
	\init_{w_\sigma} f_{i_1}, \dots, \init_{w_\sigma} f_{i_{n-d}} \in k[M].
\]
Thus this preimage is equal to $\init_{w_\sigma} X$ in $T$. Under the identification of $T$ with $T_\sigma \times_k \cO_\sigma$, this preimage is equal to $T_\sigma \times_k (X^\Delta \cap \cO_\sigma)$, and we are done.
\end{proof}

We can now prove the next part of \autoref*{tropcompequivariantclass}.

\begin{proposition}
Let $\sigma \in \Delta$. Then $\init_{w_\sigma}X$ is invariant under the $G$-action on $T$.
\end{proposition}

\begin{proof}
This follows from \autoref*{XDeltastratuminvariant}, \autoref*{twoactionsfororbitsarethesame}, and \autoref*{writinginitialformasproduct}.
\end{proof}

For the remainder of this section, endow $X^\Delta$ with the $G$-action induced by restriction of the $G$-action of $Y(\Delta)$, and for each $\sigma \in \Delta$, endow $\init_{w_\sigma}X$ with the $G$-action induced by the restriction of the $G$-action on $T$ and endow $X^\Delta \cap \cO_\sigma$ with the $G$-action induced by the restriction of the $G$-action on $Y(\Delta)$.

\begin{proposition}
\label{actioninitialandboundarycomparison}
Let $\sigma \in \Delta$. Then
\[
	[\init_{w_\sigma}X, G] = (\bL - 1)^{\dim \sigma} [X^\Delta \cap \cO_\sigma, G] \in K_0^G(\Var_k).
\]
\end{proposition}

\begin{proof}
Let $G \to T_\sigma$ be the composition of $G\to T$ and the projection $T \to T_\sigma$ induced by the indentification of $T$ with $T_\sigma \times_k \cO_\sigma$. Endow $T_\sigma$ with the $G$-action induced by $G \to T_\sigma$ and the action of $T_\sigma$ on itself given by left multiplication. By \autoref*{twoactionsfororbitsarethesame} and \autoref*{writinginitialformasproduct},
\[
	[\init_{w_\sigma}X, G] = [T_\sigma, G] [X^\Delta \cap \cO_\sigma, G] \in K_0^G(\Var_k).
\]
Then by \autoref*{torusequivariantclasstrivial},
\[
	[\init_{w_\sigma}X, G] = (\bL - 1)^{\dim \sigma} [X^\Delta \cap \cO_\sigma, G] \in K_0^G(\Var_k).
\]
\end{proof}

The next proposition completes the proof of \autoref*{tropcompequivariantclass}.

\begin{proposition}
We have that
\[
	(\bL-1)^{\dim \Delta} [X^\Delta, G] = \sum_{\sigma \in \Delta} (\bL-1)^{\dim \Delta - \dim \sigma} [\init_{w_\sigma} X, G] \in K_0^G(\Var_k).
\]
\end{proposition}

\begin{proof}
We have that
\[
	[X^\Delta, G] = \sum_{\sigma \in \Delta} [X^\Delta \cap \cO_\sigma, G] \in K_0^G(\Var_k),
\]
so by \autoref*{actioninitialandboundarycomparison},
\begin{align*}
	(\bL-1)^{\dim \Delta} [X^\Delta, G] &= \sum_{\sigma \in \Delta}(\bL-1)^{\dim \Delta - \dim \sigma} (\bL-1)^{\dim \sigma} [X^\Delta \cap \cO_\sigma, G]\\
	&= \sum_{\sigma \in \Delta} (\bL-1)^{\dim \Delta - \dim \sigma} [\init_{w_\sigma} X, G] \in K_0^G(\Var_k).
\end{align*}
\end{proof}

\section{Sch\"{o}n compactifications in families}
\label{schoncompactificationsinfamilies}

In this section we will show that under certain conditions, sch\"{o}n compactifications can be constructed in families.

Let $S = \Spec(A)$ be a nonempty connected smooth finite type scheme over $k$, and for each $s \in S(k)$ and $f \in A[M]$, let $f(s) \in k[M]$ denote the restriction of $f$ to $T = T \times_k \{s\} \subset T \times_k S = \Spec(A[M])$.

Let $f_1, \dots, f_m \in A[M] \setminus \{0\}$ be such that for each $i \in \{1, \dots, m\}$,
\[
	f_i = \sum_{u \in M} a^{(i)}_u \chi^u,
\]
where each $a^{(i)}_u \in A^\times \cup \{0\}$. Let $d \in \Z_{>0}$, and let $\Delta, \overline{\Delta}$ be fans in $N$.

Let $X \hookrightarrow T \times_k S$ be the closed subscheme defined by the ideal generated by $f_1, \dots, f_m$, for each $s \in S(k)$, let $X_s \hookrightarrow T$ denote the fiber of $X$ over $s$, and suppose that the following hypotheses hold.

\begin{enumerate}[(i)]

\item For all $s \in S(k)$, $X_s$ is pure dimension $d$ and sch\"{o}n in $T$.

\item $\Delta \subset \overline{\Delta}$.

\item $\overline{\Delta}$ is unimodular.

\item \label{familytropicalfanhypothesis} For all $s \in S(k)$, we have that $\Trop(X_s)$ is equal to the support of $\Delta$.

\item \label{familyinitialdegenerationsconstantinterior} For each $\sigma \in \overline{\Delta}$, each pair $w_1, w_2 \in \relint(\sigma)$, each $i \in \{1, \dots, m\}$, and each $s \in S(k)$,
\[
	\init_{w_1}(f_i(s)) = \init_{w_2}(f_i(s)).
\]

\item \label{familygrobnerbasisoneachcone}For each $\sigma \in \overline{\Delta}$, there exist $i_1, \dots, i_{n-d} \in \{1, \dots, m\}$ such that for all $s \in S(k)$, we have that $f_{i_1}(s), \dots, f_{i_{n-d}}(s) \in k[M]$ generate the ideal defining $X_s$ in $T$ and $\init_w(f_{i_1}(s)), \dots, \init_w(f_{i_{n-d}}(s))$ generate the ideal defining $\init_w (X_s)$ in $T$ for all $w \in \sigma$.

\end{enumerate}

Let $Y(\Delta), Y(\overline{\Delta})$ be the $T$-toric varieties defined by $\Delta, \overline{\Delta}$, respectively, and for each $s \in S(k)$, let $X_s^\Delta \hookrightarrow Y(\Delta)$ be the closure of $X_s$.

\begin{proposition}
\label{schoncompinfamily}
There exists a closed subscheme $X^{\overline{\Delta}} \hookrightarrow Y(\overline{\Delta}) \times_k S$ that is smooth over $S$ and such that for each $s \in S(k)$, the fiber of $X^{\overline{\Delta}}$ over $s$ is equal to $X_s^\Delta$ as a closed subscheme of $Y(\Delta) \subset Y(\overline{\Delta})$.
\end{proposition}

\subsection{Proof of \autoref*{schoncompinfamily}} We begin our proof with the following lemma.

\begin{lemma}
Let $\sigma \in \overline{\Delta}$. For each $i \in \{1, \dots, m\}$, there exists $u^\sigma_i \in M$ such that for all $s \in S(k)$ and all $w \in \relint(\sigma)$,
\[
	-u^\sigma_i \in \supp(\init_w(f_i(s))).
\]
\end{lemma}

\begin{proof}
By the fact that each $a^{(i)}_u \in A^\times \cup \{0\}$ and by hypothesis (\ref*{familyinitialdegenerationsconstantinterior}), we have that for all $i \in \{1, \dots, m\}$, all $w_1, w_2 \in \relint(\sigma)$, and all $s_1, s_2 \in S(k)$,
\[
	\supp(\init_{w_1}(f_i(s_1))) = \supp(\init_{w_2}(f_i(s_2))).
\]
Thus we only need to show that there exists $u^\sigma_i \in M$ such that
\[
	-u^\sigma_i \in \supp(\init_w(f_i(s)))
\]
for some $s \in S(k)$ and some $w \in \relint(\sigma)$.
Fix some $w \in \relint(\sigma)$, and because $S$ is nonempty, we can fix some $s \in S(k)$. Because each $f_i \neq 0$, the hypothesis on each $a^{(i)}_u$ implies that each $f_i(s) \neq 0$, and therefore each $\supp(\init_w(f_i(s))) \neq \emptyset$. The lemma follows.
\end{proof}

For the remainder of this section, for each $\sigma \in \overline{\Delta}$ and each $i \in \{1, \dots, m\}$, fix some $u^\sigma_i \in M$ such that for all $s \in S(k)$ and all $w \in \relint(\sigma)$,
\[
	-u^\sigma_i \in \supp(\init_w(f_i(s))).
\]

By hypothesis (\ref*{familygrobnerbasisoneachcone}), for each $\sigma \in \overline{\Delta}$, we may fix $i^\sigma_1, \dots, i^\sigma_{n-d} \in \{1, \dots, m\}$ such that for all $s \in S(k)$, we have that $f_{i^\sigma_1}(s), \dots, f_{i^\sigma_{n-d}}(s) \in k[M]$ generate the ideal defining $X_s$ in $T$ and $\init_w(f_{i^\sigma_1}(s)), \dots, \init_w(f_{i^\sigma_{n-d}}(s))$ generate the ideal defining $\init_w (X_s)$ in $T$ for all $w \in \sigma$.

For each $\sigma \in \overline{\Delta}$, let $Y(\sigma) = \Spec(k[\sigma^\vee \cap M])$ be the affine $T$-toric variety defined by $\sigma$, and for each $s \in S(k)$, let $X_s^\sigma \hookrightarrow Y(\sigma)$ be the closure of $X_s$.

\begin{proposition}
\label{functionscutoutcorrectfibersforcone}
Let $\sigma \in \overline{\Delta}$. Then $\chi^{u^\sigma_{i^\sigma_1}}f_{i^\sigma_1}, \dots, \chi^{u^\sigma_{i^\sigma_{n-d}}}f_{i^\sigma_{n-d}} \in A[\sigma^\vee \cap M]$, and for all $s \in S(k)$, we have that $\chi^{u^\sigma_{i^\sigma_1}}f_{i^\sigma_1}(s), \dots, \chi^{u^\sigma_{i^\sigma_{n-d}}}f_{i^\sigma_{n-d}}(s) \in k[\sigma^\vee \cap M]$ generate the ideal defining $X_s^\sigma$ in $Y(\sigma)$.
\end{proposition}

\begin{proof}
Let $s \in S(k)$. Because $X_s$ is sch\"{o}n, it is smooth and therefore reduced. By \autoref*{LuxtonQufansupportedontropistropicalfan} and hypothesis (\ref*{familytropicalfanhypothesis}), we have that $\Delta$ is a tropical fan for $X_s \hookrightarrow T$. Therefore the hypotheses and \autoref*{tcideal} imply that
\[
	\chi^{u^\sigma_{i^\sigma_1}}f_{i^\sigma_1}(s), \dots, \chi^{u^\sigma_{i^\sigma_{n-d}}}f_{i^\sigma_{n-d}}(s) \in k[\sigma^\vee \cap M],
\]
and $\chi^{u^\sigma_{i^\sigma_1}}f_{i^\sigma_1}(s), \dots, \chi^{u^\sigma_{i^\sigma_{n-d}}}f_{i^\sigma_{n-d}}(s)$ generate the ideal defining $X_s^\sigma$ in $Y(\sigma)$. Because each $a^{(i)}_u \in A^\times \cup \{0\}$, this also implies that
\[
	\chi^{u^\sigma_{i^\sigma_1}}f_{i^\sigma_1}, \dots, \chi^{u^\sigma_{i^\sigma_{n-d}}}f_{i^\sigma_{n-d}} \in A[\sigma^\vee \cap M],
\]
and we are done.
\end{proof}

For each $\sigma \in \overline{\Delta}$, let $X^\sigma \hookrightarrow Y(\sigma) \times_k S = \Spec(A[\sigma^\vee \cap M])$ be the closed subscheme defined by the ideal generated by $\chi^{u^\sigma_{i^\sigma_1}}f_{i^\sigma_1}, \dots, \chi^{u^\sigma_{i^\sigma_{n-d}}}f_{i^\sigma_{n-d}}$.

\begin{remark}
\label{fibersofXsigma}
By \autoref*{functionscutoutcorrectfibersforcone}, for each $\sigma \in \overline{\Delta}$ and each $s \in S(k)$, the fiber of $X^\sigma$ over $s$ is equal to $X_s^\sigma$ as a closed subscheme of $Y(\sigma)$.
\end{remark}

\begin{proposition}
\label{XsigmaissmoothoverS}
For each $\sigma \in \overline{\Delta}$, the morphism $X^\sigma \to S$ is smooth.
\end{proposition}

\begin{proof}
For all $s \in S(k)$, $X_s$ is pure dimension $d$ and sch\"{o}n in $T$, and $\overline{\Delta}$ is a unimodular fan containing a tropical fan for $X_s \hookrightarrow T$, so $X_s^\sigma$ is smooth and pure dimension $d$. Thus by \autoref*{fibersofXsigma}, each irreducible component of $X^\sigma$ has dimension at most
\[
	\dim S + d = \dim(Y(\sigma) \times_k S) - (n-d).
\]
Because $X^\sigma$ is a closed subscheme of $Y(\sigma) \times_k S$ defined by $n-d$ functions, each irreducible component of $X^\sigma$ has dimension at least $\dim(Y(\sigma) \times_k S) - (n-d)$, so $X^\sigma$ is a complete intersection in a smooth variety and is thus Cohen-Macaulay, and
\[
	d = \dim X^\sigma - \dim S.
\]
Then by \autoref*{fibersofXsigma}, the fact that each $X_s^\sigma$ is smooth and pure dimension $d$, and the fact that $S$ is smooth, we have that $X^\sigma \to S$ is smooth.
\end{proof}

In the remainder of this section, we will define $X^{\overline{\Delta}} \hookrightarrow Y(\overline{\Delta}) \times_k S$, show that it is glued together from the $X^\sigma \hookrightarrow Y(\sigma) \times_k S$, and obtain a proof of \autoref*{schoncompinfamily}.

There is an ideal sheaf on $Y(\overline{\Delta}) \times_k S$ such that on each $Y(\sigma) \times_k S$, the ideal sheaf is given by
\[
	( \chi^{u} f_i \, | \, \text{$i \in \{1, \dots, m\}$ and $u \in M$ such that $\chi^u f_i \in A[\sigma^\vee \cap M]$}) \subset A[\sigma^\vee \cap M].
\]
Let $X^{\overline{\Delta}} \hookrightarrow Y(\overline{\Delta}) \times_k S$ be the closed subscheme defined by this ideal sheaf.

\begin{proposition}
\label{XoverlinedeltaislocallyXsigma}
For each $\sigma \in \overline{\Delta}$, we have that $X^{\overline{\Delta}} \cap (Y(\sigma) \times_k S) = X^\sigma$ as closed subschemes of $Y(\sigma) \times_k S$.
\end{proposition}

\begin{proof}
Let $\sigma \in \overline{\Delta}$. We first note that by construction, $X^{\overline{\Delta}} \cap (Y(\sigma) \times_k S)$ is a closed subscheme of $X^\sigma$. We will next show that $X^{\overline{\Delta}} \cap (Y(\sigma) \times_k S)$ and $X^\sigma$ have the same support. To do this, it will be sufficient to show that for each $s \in S(k)$, the fiber of $X^{\overline{\Delta}} \cap (Y(\sigma) \times_k S)$ over $s$ is equal to the fiber of $X^\sigma$ over $s$.

Let $s \in S(k)$. By construction, $X_s^\sigma$ is a closed subscheme of the fiber of $X^{\overline{\Delta}} \cap (Y(\sigma) \times_k S)$ over $s$. Also, the fiber of $X^{\overline{\Delta}} \cap (Y(\sigma) \times_k S)$ over $s$ is a closed subscheme of the fiber of $X^\sigma$ over $s$, which by \autoref*{fibersofXsigma}, is equal to $X_s^\sigma$. Thus the fiber of $X^{\overline{\Delta}} \cap (Y(\sigma) \times_k S)$ over $s$ is equal to the fiber of $X^\sigma$ over $s$.

Therefore, $X^{\overline{\Delta}} \cap (Y(\sigma) \times_k S)$ is a closed subscheme of $X^\sigma$ that is supported on all of $X^\sigma$. By \autoref*{XsigmaissmoothoverS} and the fact that $S$ is smooth over $k$, we have that $X^\sigma$ is smooth over $k$, and in particular, $X^\sigma$ is reduced. Thus we are done.
\end{proof}

The following corollary completes the proof of \autoref*{schoncompinfamily}.

\begin{corollary}
The morphism $X^{\overline{\Delta}} \to S$ is smooth, and for each $s \in S(k)$, the fiber of $X^{\overline{\Delta}}$ over $s$ is equal to $X_s^\Delta$ as a closed subscheme of $Y(\Delta) \subset Y(\overline{\Delta})$.
\end{corollary}

\begin{proof}
The fact that $X^{\overline{\Delta}} \to S$ is smooth follows immediately from Propositions \ref*{XsigmaissmoothoverS} and \ref*{XoverlinedeltaislocallyXsigma}. Let $s \in S(k)$, and let $X_s^{\overline{\Delta}} \hookrightarrow Y(\overline{\Delta})$ be the closure of $X_s$. Then \autoref*{fibersofXsigma} and \autoref*{XoverlinedeltaislocallyXsigma} imply that $X_s^{\overline{\Delta}}$ is equal to the fiber of $X^{\overline{\Delta}}$ over $s$. Because $\Delta$ is a tropical fan for $X_s \hookrightarrow T$, we have that $X_s^\Delta$ is proper over $k$. Thus $X_s^\Delta = X_s^{\overline{\Delta}}$, and we are done.
\end{proof}

\section{Additive invariants of the Milnor fiber}
\label{additiveinvariantsMilnorfibersection}

Let $d, n \in \Z_{>0}$, let $\cM$ be a rank $d$ loop-free matroid on $\{1, \dots, n\}$, and assume that the characteristic of $k$ does not divide $n$. In this section, we will prove \autoref*{additiveinvariantsmilnorfiber}. 

Let $\mu_n \to \bG_{m,k}^n$ be the algebraic group homomorphism given by composing the inclusion $\mu_n \hookrightarrow \bG_{m,k}$ with the diagonal morphism $\bG_{m,k} \to \bG_{m,k}^n$. Endow $\bG_{m,k}^n$ with the $\mu_n$-action induced by $\mu_n \to \bG_{m,k}^n$ and the action of $\bG_{m,k}^n$ on itself given by left multiplication. 

\begin{remark}
Each $\xi \in \mu_n$ acts on $\bG_{m,k}^n$ by scalar multiplication, so for each $\cA \in \Gr_\cM(k)$, the Milnor fiber $F_\cA$ is invariant under the $\mu_n$-action on $\bG_{m,k}^n$, and the $\mu_n$-action on $F_\cA$ is equal to the restriction of the $\mu_n$-action on $\bG_{m,k}^n$.
\end{remark}

We now show the existence of fans that will be used in proving \autoref*{additiveinvariantsmilnorfiber}.

\begin{lemma}
\label{constructingmilnorcompactificationfan}
There exist fans $\Delta, \overline{\Delta}$ in $\Z^n$ such that

\begin{enumerate}[(i)]

\item $\Delta \subset \overline{\Delta}$,

\item the $\bG_{m,k}^n$-toric variety defined by $\overline{\Delta}$ is smooth and projective,

\item the support of $\Delta$ is equal to $\Trop(\cM) \cap (1, \dots, 1)^\perp$,

\item \label{contructingmilnorcompactificationfanmonomialminus1}for each $\sigma \in \overline{\Delta}$ and each pair $w_1, w_2 \in \relint(\sigma)$,
\[
	\init_{w_1}(x_1 \cdots x_n - 1) = \init_{w_2}(x_1 \cdots x_n - 1),
\]

\item \label{constructingmilnorcompactificationfancircuitforms}for each $\sigma \in \overline{\Delta}$, each pair $w_1, w_2 \in \relint(\sigma)$, each circuit $C$ in $\cM$, and each $\cA \in \Gr_\cM(k)$,
\[
	\init_{w_1} L_C^\cA = \init_{w_2} L_C^\cA,
\]
and

\item \label{constructingmilnorcompactificationfanwmaximal}for each $\sigma \in \overline{\Delta}$, there exists $B \in \cB(\cM)$ such that for all $\cA \in \Gr_\cM(k)$, we have that $\{L_{C(\cM, i, B)}^\cA \, | \, i \in \{1, \dots, n\} \setminus B\} \cup \{x_1 \cdots x_n -1\}$ is a generating set for the ideal defining $F_\cA$ in $\bG_{m,k}^n$ and $\{\init_w L_{C(\cM, i, B)}^\cA \, | \, i \in \{1, \dots, n\} \setminus B\} \cup \{\init_w(x_1 \cdots x_n -1)\}$ is a generating set for the ideal defining $\init_w F_\cA$ in $\bG_{m,k}^n$ for all $w \in \sigma$.

\end{enumerate}
\end{lemma}

\begin{proof}
Let $\Sigma$ be a fan in $\Z^n$ supported on $\Trop(\cM) \cap (1, \dots, 1)^\perp$. There exists a fan $\overline{\Sigma}$ in $\Z^n$ supported on $\R^n$ and containing $\Sigma$. See for example \cite{Ewald}[III. Theorem 2.8]. Let $\Sigma_1$ be the dual fan to the Newton polyhedron of $(x_1 \cdots x_n - 1)$. For each circuit $C$ in $\cM$, let $\Sigma_C$ be the dual fan to the Newton polyhedron of $L_C^\cA$ for some, or equivalently for all, $\cA \in \Gr_\cM(k)$. Let $\Sigma_2$ be a fan in $\Z^n$, supported on $\R^n$, such that the function $(w_1, \dots, w_n) \mapsto \max_{B \in \cM} \sum_{i \in B} w_i$ is linear on each cone in $\Sigma_2$.

Now let $\overline{\Delta}_1$ be the common refinement of $\overline{\Sigma}, \Sigma_1, \Sigma_2$, and $\Sigma_C$ for each circuit $C$ in $\cM$. By construction, $\overline{\Delta}_1$ satisfies (\ref*{contructingmilnorcompactificationfanmonomialminus1}) and (\ref*{constructingmilnorcompactificationfancircuitforms}) and contains a fan supported on $\Trop(\cM) \cap (1, \dots, 1)^\perp$. By the choice of $\Sigma_2$, for each $\sigma \in \overline{\Delta}_1$,
\[
	\bigcap_{w \in \sigma} \cB(\cM_w) \neq \emptyset.
\]
Thus by \autoref*{grobnercompleteintersectionbasismilnorfiber}, we have that $\overline{\Delta}_1$ satisfies (\ref*{constructingmilnorcompactificationfanwmaximal}).

By construction, $\overline{\Delta}_1$ is supported on $\R^n$. Thus by toric Chow's lemma and toric resolution of singularities, there exists a refinement $\overline{\Delta}$ of $\overline{\Delta}_1$ whose associated toric variety is smooth and projective. Letting $\Delta$ be the subfan of $\overline{\Delta}$ that is supported on $\Trop(\cM) \cap (1, \dots, 1)^\perp$, we see that we are done.
\end{proof}

Let $\Delta, \overline{\Delta}$ be fans in $\Z^n$ satisfying the conclusion of \autoref*{constructingmilnorcompactificationfan}. Let $Y(\Delta)$ (resp. $Y(\overline{\Delta})$) be the $\bG_{m,k}^n$-toric variety defined by $\Delta$ (resp. $\overline{\Delta}$), and endow it with the $\mu_n$-action induced by $\mu_n \to \bG_{m,k}^n$ and the $\bG_{m,k}^n$-action on $Y(\Delta)$ (resp. $Y(\overline{\Delta})$).

For each $\cA \in \Gr_\cM(k)$, let $F_\cA^\Delta \hookrightarrow Y(\Delta)$ be the closure of $F_\cA$. By \autoref*{tropcompequivariantclass} and our choice of $\Delta$, we see that each $F_\cA^\Delta$ is invariant under the $\mu_n$-action on $Y(\Delta)$, so we endow each $F_\cA^\Delta$ with the $\mu_n$-action given by restriction of the $\mu_n$-action on $Y(\Delta)$.

\begin{proposition}
\label{initialdegenerationscompactificationsequivariantclasses}
Let $\cA \in \Gr_\cM(k)$, and for each $\sigma \in \Delta$, choose some $w_\sigma \in \relint(\sigma)$. Then
\[
	(\bL-1)^{\dim \Delta} [F_\cA^\Delta, \mu_n] = \sum_{\sigma \in \Delta} (\bL-1)^{\dim \Delta - \dim \sigma} [F_{\cA_{w_\sigma}}, \mu_n] \in K_0^{\mu_n}(\Var_k).
\]
\end{proposition}

\begin{proof}
This follows immediately from \autoref*{tropcompequivariantclass}, our choice of $\Delta$, and the fact that $\init_w F_\cA = F_{\cA_w}$ for all $w \in \Trop(\cM) \cap (1, \dots, 1)^\perp$.
\end{proof}

In the following lemmas, we construct families of linear subspaces of type $\cM$.

\begin{lemma}
\label{universalcircuitforms}
There exists a set $\{a_i^C\}$ of units on $\Gr_\cM$, indexed by circuits $C$ in $\cM$ and elements $i \in C$, such that if $X \hookrightarrow \bA_k^n \times_k \Gr_\cM$ is the closed subscheme defined by the ideal generated by
\[
	\{\sum_{i \in C} a_i^Cx_i \, | \, \text{$C$ a circuit in $\cM$}\},
\]
then for each $\cA \in \Gr_\cM(k)$, the fiber of $X$ over $\cA$ is equal to $X_\cA$ as a closed subscheme of $\bA_k^n$.
\end{lemma}

\begin{proof}
For each regular function $a$ on $\Gr_\cM$ and each $\cA \in \Gr_\cM(k)$, let $a(\cA) \in k$ denote the evaluation of $a$ at $\cA$. It is sufficient to show that if $C$ is a circuit in $\cM$, then there exists a set $\{a_i^C\}$ of units on $\Gr_\cM$, indexed by $i \in C$, such that $\sum_{i \in C} a_i^C(\cA) x_i$ is equal to a nonzero scalar multiple of $L_C^\cA$ for all $\cA \in \Gr_\cM(k)$.

Let $C$ be a circuit in $\cM$. Let $B \in \cB(\cM)$ and $j \in \{1, \dots, n\} \setminus B$ be such that $C = C(\cM, j, B)$. Let $\Gr_{d,n} \hookrightarrow \Proj(k[y_I \, | \, I \subset \{1, \dots, n\}, \# I = d])$ be the Pl\"{u}cker embedding of $\Gr_{d,n}$. For each $I \subset \{1, \dots, n\}$ with $\# I = d$, set $b_I$ to be the restriction of $y_I/y_B$ to $\Gr_\cM$. Then $b_I$ is a unit if $I \in \cB(\cM)$, and $b_I = 0$ otherwise. Write $B \cup \{j\} = \{i_1, \dots, i_{d+1}\}$ with $i_1 < i_2 < \dots < i_{d+1}$. For each $\ell \in \{1, \dots, d+1\}$, set
\[
	a_{i_\ell}^C = (-1)^\ell b_{(B \cup \{j\}) \setminus \{i_\ell\}}.
\]
By \autoref*{fundamentalcircuitbasisexchange}, $a_i^C$ is a unit for all $i \in C$, and $a_i^C = 0$ for all $i \in (B \cup \{j\}) \setminus C$. By construction, for all $\cA \in \Gr_\cM(k)$,
\[
	\sum_{i \in C} a_i^C(\cA) x_i = \sum_{\ell \in \{1, \dots, d+1\}} (-1)^\ell b_{(B \cup \{j\}) \setminus \{i_\ell\}}(\cA)x_{i_\ell}
\]
is in the ideal defining $X_\cA$ in $\bA_k^n$. Thus $\sum_{i \in C} a_i^C(\cA) x_i$ is a nonzero scalar multiple of $L_C^\cA$, and we are done.
\end{proof}

\begin{lemma}
\label{familyofmilnorfibers}
Let $\cA_1, \cA_2 \in \Gr_\cM(k)$ be in the same irreducible component of $\Gr_\cM$. Then there exists an affine nonempty connected smooth finite type scheme $S = \Spec(A)$ over $k$ and a subset $\{a_i^C\} \subset A^\times$, indexed by circuits $C$ in $\cM$ and elements $i \in C$, such that if $X \hookrightarrow \bA_k^n \times_k S$ is the closed subscheme defined by the ideal generated by
\[
	\{\sum_{i \in C} a_i^C x_i \, | \, \text{$C$ a cricuit in $\cM$}\} \subset A[x_1, \dots, x_n],
\]
then
\begin{enumerate}[(i)]

\item for each $s \in S(k)$, there exists $\cA \in \Gr_\cM(k)$ such that the fiber of $X$ over $s$ is equal to $X_\cA$ as a closed subscheme of $\bA_k^n$, and

\item there exist $s_1, s_2 \in S(k)$ such that the fiber of $X$ over $s_1$ (resp. $s_2$) is equal to $X_{\cA_1}$ (resp. $X_{\cA_2}$) as a closed subscheme of $\bA_k^n$.

\end{enumerate}
\end{lemma}

\begin{proof}
We will show that there exists an affine nonempty connected smooth finite type scheme $S = \Spec(A)$ over $k$ and a morphism $S \to \Gr_\cM$ whose image contains both $\cA_1$ and $\cA_2$. If such a morphism exists, then we may set $\{a_i^C\}$ to be the pullbacks of the units shown to exist in \autoref*{universalcircuitforms}, and our proof would be done.

Let $S_1$ be an integral curve inside $\Gr_\cM$ containing $\cA_1$ and $\cA_2$. Such a curve exists, for example, by \cite[Corollary 1.9]{CharlesPoonen}. Let $S_2$ be the normalization of $S_1$, and let $S_2 \to \Gr_\cM$ be the composition of the normalization map with the inclusion of $S_1$ into $\Gr_\cM$. Then $S_2$ is a smooth connected curve and $\cA_1, \cA_2$ are in the image of $S_2 \to \Gr_\cM$. Now let $S$ be an open affine in $S_2$ containing points that get mapped to $\cA_1$ and $\cA_2$, and we are done.
\end{proof}

For the remainder of this section, let $P$ be a $\Z[\bL]$-module and let $\nu: K_0^{\mu_n}[\Var_k] \to P$ be a $\Z[\bL]$-module morphism that is constant on smooth projective families with $\mu_n$-action.

\begin{proposition}
\label{samecomponentschubertcellsameadditiveinvariantforcompactification}
Let $\cA_1, \cA_2 \in \Gr_\cM(k)$ be in the same connected component of $\Gr_\cM$. Then
\[
	\nu[F_{\cA_1}^\Delta, \mu_n] = \nu[F_{\cA_2}^\Delta, \mu_n].
\]
\end{proposition}

\begin{proof}
We may assume that $\cA_1$ and $\cA_2$ are in the same irreducible component of $\Gr_\cM$. Then we may let $S= \Spec(A)$ and $\{a_i^C\}$ be as in the conclusion of \autoref*{familyofmilnorfibers}. For each circuit $C$ in $\cM$, set
\[
	L_C = \sum_{i \in C} a_i^Cx_i \in A[x_1, \dots, x_n],
\]
and let $X \hookrightarrow \bA_k^n \times_k S$ be the closed subscheme defined by the ideal generated by
\[
	\{L_C \, | \, \text{$C$ a circuit in $\cM$}\} \subset A[x_1, \dots, x_n].
\]
For each $s \in S(k)$, let $\cA_s \in \Gr_\cM(k)$ be such that the fiber of $X$ over $s$ is equal to $X_{\cA_s}$ as a closed subscheme of $\bA_k^n$.

For each $s \in S(k)$ and each circuit $C$ in $\cM$, the restriction of $L_C$ to the fiber of $\bA_k^n \times_k S$ over $s$ is a linear form in the ideal of $X_{\cA_s}$ supported on the coordinates indexed by $C$, so this restriction is equal to a nonzero scalar multiple of $L_C^{\cA_s}$.

Now let $F \hookrightarrow \bG_{m,k}^n \times_k S$ be the closed subscheme defined by the ideal generated by the set
\[
	\{L_C \, | \, \text{$C$ a circuit in $\cM$}\} \cup \{x_1 \cdots x_n - 1\} \subset A[x_1^{\pm 1}, \dots, x_n^{\pm 1}].
\]
Then for each $s \in S(k)$, the fiber of $F$ over $s$ is equal to $F_{\cA_s}$ as a closed subscheme of $\bG_{m,k}^n$. Therefore by \autoref*{schoncompinfamily} and our choice of $\Delta, \overline{\Delta}$, there exists a closed subscheme $F^{\overline{\Delta}} \hookrightarrow Y(\overline{\Delta}) \times_k S$ that is smooth over $S$ and such that for each $s \in S(k)$, the fiber of $F^{\overline{\Delta}}$ over $s$ is equal to $F_{\cA_s}^\Delta$ as a closed subscheme of $Y(\Delta) \subset Y(\overline{\Delta})$. Endow $Y(\overline{\Delta}) \times_k S$ with the diagonal $\mu_n$-action induced by the $\mu_n$-action on $Y(\Delta)$ and the trivial $\mu_n$-action on $S$. 

We will show that $F^{\overline{\Delta}}$ is invariant under the $\mu_n$-action on $Y(\overline{\Delta}) \times_k S$. Because $F^{\overline{\Delta}}$ is smooth over the smooth $k$-scheme $S$, we have that $F^{\overline{\Delta}}$ is reduced. Therefore it is sufficient to show that $F^{\overline\Delta}(k)$ is invariant under the $\mu_n$-action on $Y(\overline{\Delta}) \times_k S$. This holds because $F_{\cA_s}^\Delta$ is invariant under the $\mu_n$-action on $Y(\Delta)$ for each $s \in S(k)$. Therefore $F^{\overline\Delta}$ is invariant under the $\mu_n$-action on $Y(\overline\Delta) \times_k S$.

We now endow $F^{\overline{\Delta}}$ with the $\mu_n$-action given by restriction of the $\mu_n$-action on $Y(\overline{\Delta}) \times_k S$. By construction, the smooth projective morphism $F^{\overline\Delta} \to S$ is $\mu_n$-equivariant, where $S$ is given the trivial $\mu_n$-action.  Thus because $S$ is connected, the map
\[
	S(k) \to P: s \mapsto \nu[F_{\cA_s}^\Delta, \mu_n]
\]
is constant. By the choice of $X$ and $S$, there exist $s_1, s_2 \in S(k)$ such that $F_{\cA_{s_1}}^\Delta = F_{\cA_1}^\Delta$ and $F_{\cA_{s_2}}^\Delta = F_{\cA_2}^\Delta$, and we are done.

\end{proof}

The following proposition completes the proof of \autoref*{additiveinvariantsmilnorfiber}.

\begin{proposition}
Let $\cA_1, \cA_2 \in \Gr_\cM(k)$ be in the same connected component of $\Gr_\cM$, and suppose that $P$ is torsion-free as a $\Z[\bL]$-module. Then
\[
	\nu[F_{\cA_1}, \mu_n] = \nu[F_{\cA_2}, \mu_n].
\]
\end{proposition}

\begin{proof}
We will prove the proposition by induction on the number of bases of $\cM$. For each $\sigma \in \Delta$, choose some $w_\sigma \in \relint(\sigma)$. Note that for all $\sigma \in \Delta$, we have that $(\cA_1)_{w_\sigma}$ and $(\cA_2)_{w_\sigma}$ are in the same connected component of $\Gr_{\cM_{w_\sigma}}$. Thus by induction, for all $\sigma \in \Delta$ such that $\cM_{w_\sigma} \neq \cM$,
\[
	\nu[F_{(\cA_1)_{w_\sigma}}, \mu_n] = \nu[F_{(\cA_2)_{w_\sigma}}, \mu_n].
\]
Now set
\[
	q(\bL) = \sum_{\substack{\sigma \in \Delta\\ \cM_{w_\sigma}= \cM}} (\bL-1)^{\dim\Delta - \dim\sigma} \in \Z[\bL],
\]
set
\[
	q_\Delta(\bL) = (\bL-1)^{\dim\Delta} \in \Z[\bL],
\]
and for each $\sigma \in \Delta$, set
\[
	q_\sigma(\bL) = (\bL-1)^{\dim\Delta-\dim\sigma} \in \Z[\bL].
\]
Then by Propositions \ref*{initialdegenerationscompactificationsequivariantclasses} and \ref*{samecomponentschubertcellsameadditiveinvariantforcompactification},
\begin{align*}
	q(\bL) \cdot  \nu[F_{\cA_1}, \mu_n] &= q_\Delta(\bL) \cdot \nu[F_{\cA_1}^\Delta, \mu_n] - \sum_{\substack{\sigma \in \Delta\\ \cM_{w_\sigma} \neq \cM}} q_\sigma(\bL) \cdot \nu[F_{(\cA_1)_{w_\sigma}}, \mu_n]\\
	&= q_\Delta(\bL) \cdot \nu[F_{\cA_2}^\Delta, \mu_n] - \sum_{\substack{\sigma \in \Delta\\ \cM_{w_\sigma} \neq \cM}} q_\sigma(\bL) \cdot \nu[F_{(\cA_2)_{w_\sigma}}, \mu_n]\\
	&= q(\bL) \cdot \nu[F_{\cA_2}, \mu_n].
\end{align*}
By its definition, we see that $q(\bL) \neq 0$. Therefore,
\[
	\nu[F_{\cA_1}, \mu_n] = \nu[F_{\cA_2}, \mu_n].
\]
\end{proof}

\bibliographystyle{alpha}
\bibliography{HAMHNMF}

\end{document}